\documentclass[11pt]{amsart}

\usepackage{amssymb}
\usepackage{amsthm}
\usepackage{amsmath}
\usepackage{graphicx}
\usepackage{amsaddr}
\usepackage{comment}
\usepackage[usenames,dvipsnames]{xcolor}
\usepackage[margin = 1in]{geometry}
\usepackage{enumerate}
\usepackage[toc,page]{appendix}
\usepackage{lineno}
\usepackage{color}
\usepackage{hyperref}
\usepackage{mhchem}
\usepackage{siunitx}

\newcommand{\blue}{}
\definecolor{mygreen}{rgb}{0.1,0.75,0.2}

\providecommand{\bbs}[1]{\left(#1\right)}

 \newtheorem{thm}{Theorem}[section]
 \newtheorem{cor}[thm]{Corollary}
 \newtheorem{lem}[thm]{Lemma}
 \newtheorem{prop}[thm]{Proposition}

 \theoremstyle{definition}
 \newtheorem{defn}{Definition}

 \theoremstyle{remark}
 \newtheorem{rem}{Remark}

 \numberwithin{equation}{section}

\newcommand{\la}{\langle}
\newcommand{\ra}{\rangle}
\newcommand{\pt}{\partial}

\newcommand{\eps}{\varepsilon}

\newcommand{\ud}{\,\mathrm{d}}
\newcommand{\8}{\infty}

\newcommand{\bR}{\mathbb{R}}
\newcommand{\bZ}{\mathbb{Z}}

\newcommand{\bP}{\mathbb{P}}
\newcommand{\bS}{\mathbb{S}}

\newcommand{\vg}{\gamma}

\begin{document}

\title[Selection principle for weak KAM solutions]{A selection principle for weak KAM solutions via Freidlin-Wentzell  large deviation principle of invariant measures}

\author[Y. Gao]{Yuan Gao}
\address{Department of Mathematics, Purdue University, West Lafayette, IN}
\email{gao662@purdue.edu}

\author[J.-G. Liu]{Jian-Guo Liu}
\address{Department of Mathematics and Department of
  Physics, Duke University, Durham, NC}
\email{jliu@math.duke.edu}




\date{\today}

\maketitle

\begin{abstract}
This paper reinterprets Freidlin-Wentzell's variational construction of the rate function in the large deviation principle for invariant measures from the weak KAM perspective. Through a one-dimensional irreversible diffusion process on a torus, we explicitly characterize essential concepts in the weak KAM theory, such as the Peierls barrier and the projected Mather/Aubry/Ma\~n\'e sets.   The weak KAM representation of Freidlin-Wentzell's variational construction of the rate function is discussed based on the global adjustment for the boundary data on the Aubry set and the local trimming from the lifted Peierls barriers. This rate function gives the maximal Lipschitz continuous viscosity solution to the corresponding stationary Hamilton-Jacobi equation (HJE), satisfying Freidlin-Wentzell's variational formula for the boundary data. Choosing meaningful self-consistent boundary data at each local attractor are essential to  select a unique weak KAM solution to stationary HJE. This selected viscosity solution also serves as the global energy landscape of the original stochastic process.  This selection for stationary HJEs can be described by first taking the long time limit  and then taking the zero noise limit, which also provides a special construction of vanishing viscosity approximation.
\end{abstract}

\section{Introduction}

 The classical Kolmogorov-Arnold-Moser (KAM) theorem studied the existence of quasi-periodic solutions of a perturbed integral system. A canonical map converting original phase variables $(p,x)$ to action-angle variables $(P,X)$ can be used to transform a perturbed Hamiltonian dynamics into a (nearly) integrable system in terms of the action-angle variable. This is a classical way to study the perturbed Hamiltonian dynamics, pioneered by \textsc{Kolmogorov, Arnold, and Moser}. The canonical map is globally implicitly solved from a generating function $u(P,x)$, which solves an associated stationary Hamilton--Jacobi equation (HJE) for each action variable $P$ in the classical sense \cite{Arnold, evans2008weak}. This procedure is, in general, very hard and can only be taken in a small perturbation way. For general Hamiltonian systems far away from an integrable one, the Aubry-Mather theory developed by \textsc{Aubry} \cite{aubry1983twist}, \textsc{Mather} \cite{Mather1982} introduced various action minimizing sets and effective Hamiltonian $\bar{H}(P)$ for the corresponding Lagrangian dynamics to obtain a global understanding of the general Hamiltonian dynamics. Instead of finding the classical solution to stationary HJEs with an effective Hamiltonian $\bar{H}(P)$, the notion of a global non-differential solution defined in the viscosity sense was introduced by \textsc{Crandall and Lions} \cite{crandall1983viscosity}. Solving the family of stationary HJEs with an effective Hamiltonian in the viscosity sense has many important applications, for instance, the cell problem for the homogenization theory by \textsc{Lions, Papanicolaou, and Varadhan} \cite{lions1986homogenization} in the late 80's.

The above celebrated results on stationary HJEs lead to the development of the weak KAM theory, pioneered particularly by \textsc{Fathi} \cite{fathi1997theoreme, Fathi1998SurLC} and \textsc{Ma\~n\'e} \cite{mane1996generic}, \textsc{E} \cite{weinan1999aubry}. It is well known that solutions to dynamic HJEs can be represented in terms of the Lax--Oleinik variational formula, which computes the least action of the corresponding Lagrangian at a finite time. The weak KAM theorem by Fathi \cite{Fathi1998SurLC, fathi2008weak} proved convergence from the Lax--Oleinik semigroup representation for the dynamic HJE to a variational representation of solutions to the stationary HJE, known as weak KAM solutions. In other words, weak KAM solutions are invariant solutions for the Lax--Oleinik semigroup; see \eqref{res}.
 This variational representation   for stationary solutions to HJE uses the Ma\~n\'e potential \eqref{value} to compute the least action path in an undefined time horizon. Particularly, if one solves the least action problem (backward characteristic of the Hamiltonian dynamics) as $t\to -\infty$, which tracks back to some invariant sets of the Hamiltonian dynamics, then the Ma\~n\'e potential becomes the so-called Peierls barrier with an infinite time horizon. More importantly, those backward characteristics obtained through PDE methods can be used to characterize invariant sets in the Aubry-Mather theory for the original dynamical system.

Finding the integrable structure and characterizing those invariant sets of the original dynamics via the variational representation of the globally defined stationary solution to HJE is the central idea in the development of the weak KAM theory and thus is viewed as a generalization of KAM theory in terms of the "Hamilton-Jacobi methods". By using the concept of the projected Aubry set, the stationary variational representation only relies on the boundary values of the solution $W(x_j)$ on the projected Aubry set $\mathcal{A}$ and the Peierls barriers
\begin{equation}\label{xxx}
W(x) = \min_{x_j\in \text{invariant states}} { W(x_j) + h(x; x_j)}.
\end{equation}

This variational representation formula \eqref{xxx} is indeed already derived in 1969 in the Freidlin-Wentzell theory for the large deviation principle in the zero noise limit of the invariant measures for diffusion processes \cite{venttsel1969limiting, ventsel1970small}. The Freidlin-Wentzell theory comprehensively studied the global quasi-potentials that are globally defined and are later called the Ma\~n\'e potentials during the development of the weak KAM theory in the late 90's. The local quasi-potentials within the basin of attraction of each stable state are widely used in computing the barriers for exit problems of a stochastic dynamics. Using quasi-potentials for each basin of attraction of stable states, the Peierls barrier can be computed and can be used to construct the rate function $W(x)$ for the large deviation principle of invariant measures for those stochastic processes cf. \cite[Chapter 6, Theorem 4.3]{FW}.

This paper focuses on reinterpreting Freidlin-Wentzell's variational construction of the rate function for invariant measures in the large deviation principle from the weak KAM perspective. Through a simple one-dimensional irreversible diffusion process on a torus, we explicitly characterize all the essential concepts in the weak KAM theory, such as the projected Aubry/Mather sets, the variational representation, and the unique selection principle for boundary data on the projected Aubry set provided by the large deviation principle. These weak KAM characterizations also, in turn, help us understand the global properties of the rate function in the large deviation principle through a geometric/dynamic viewpoint and the construction of the global energy landscape, which guides the most probable path/states in the zero noise limit of the stochastic process.

We first clarify that to study the rate function for the large deviation principle, we are only interested in the critical energy level, i.e., the critical Ma\~n\'e value $c^*=0$; see Section \ref{sec3_critical}. Then the stationary HJE is
\begin{equation}
H( W'(x), x) = W'(W'-U') = 0.
\end{equation}
Here the Hamiltonian $H(p,x)=p(p-U'(x))$ can be derived from the WKB expansion for the family of invariant measures $\pi_\eps(x)= e^{-\frac{W_\eps(x)}{\eps}}$. The corresponding Lagrangian is also called the Ma\~n\'e Lagrangian
\begin{equation}
L(s,x) = \frac14(s+U'(x))^2.
\end{equation}
The most distinguishing feature is that the Ma\~n\'e Lagrangian $L(s,x)\geq 0$ and $L=0$ if and only if $s=-U'(x)$. This reduces the action minimizing path to a first-order ODE problem. Although it is not directly related to the large deviation principle, there are also other dynamics corresponding to the effective Hamiltonian $\bar{H}(P)>c^*=0$. This defines different invariant sets and action minimizing measures, which become more involved, particularly for high dimensions; cf. \cite{sorrentino2015action}.

Some results presented in this paper might be direct consequences of general results in the weak KAM theory; however, we nevertheless provide more explicit information and elementary proofs for the simple example on $\bS^1$ that are particularly connected to the Freidlin-Wentzell theory.
For a comprehensive study or survey of the weak KAM theory, we refer to Fathi's book \cite{fathi2008weak} and some very recent books \cite{sorrentino2015action, Tran21}. For recent developments of the weak KAM theory in noncompact domains, we refer to \cite{fathi2009hausdorff} for the regularity of Hamiltonians and to \cite{Wang_Wang_Yan_2019} for contact
Hamiltonian systems where $H(\nabla u(x), u(x), x)$ also depends on $u$. We particularly refer to \cite{Contreras2001, gomes2008generalized, ichihara2009long, Fathi16, ishii2020vanishing} for the weak KAM solution as a vanishing discount limit in compact/noncompact domains. The vanishing discount limit of the corresponding optimal control problem does provide a selection principle for weak KAM solutions. See also \cite{cgmt15, mitake2017selection}, which include a degenerate diffusion term in the vanishing discount limit problem, and see \cite{ITM2017} for a duality framework in the vanishing discount problem for fully nonlinear, degenerate elliptic Hamiltonians. The selection principle from the vanishing discount limit is, however, different from the selection principle provided by the large deviation principle for invariant measures. Nonuniqueness of the viscosity solution to the stationary HJE is an important issue even if the Hamiltonian is strictly convex. The nonuniqueness for the vanishing viscosity limit of both stationary HJE and stationary conservation laws are important problems. For instance, constructing a vanishing viscosity approximation to stationary HJE which has a uniform limit is still open \cite{Tran21}. For the stationary transonic flow, multiple stationary entropy shocks were constructed in Jameson et al. \cite{embid1984multiple}, and a selection principle via the stability of the time-dependent problem was studied by Embid et al. \cite{steinhoff1982multiple} and numerically computed by Shu \cite{shu1988total}. We will explain our results below.

 For a one-dimensional irreversible diffusion on torus, in Proposition \ref{prop_h}  we explicitly compute  the detailed structures of the Peierls barrier $h(y;x_i)$, which is a central concept in the definition of the projected Aubry set in the Mather-Aubry theory developed in 80's. Then we use the Peierls barrier $h(y;x_i)$ to study the detailed structure of Freidlin-Wentzell's variational construction of the rate function $W^*(x)=W(x)-\min_{x\in\bS^1} W(x)$ of the large deviation principle for invariant measures of the diffusion process on $\bS^1$. 
This includes two essential steps: (i) the global adjustment for  boundary data $W(x_i)$ at the local minimums $x_i$ of the original skew periodic potential $U(x)$; (ii) the local trimming via a variational representation for $W(x)$; see \eqref{xxx} and the local version \eqref{localW}. 

For step (i), we give an alternative proof in Lemma \ref{lem:disW} that the variational formula for the boundary data satisfies the discrete weak KAM problem. These boundary data indeed uniquely determine a maximal Lipschitz continuous viscosity solution (see Proposition \ref{prop_maximalL}) and thus the unique selection principle \eqref{wi} for these boundary data is essential to construct a global energy landscape for the original stochastic process. Particularly, when the original potential $U(x)$ is periodic itself, then we verify that the variational formula for the boundary data must give exactly the same values as the original landscape $U(x)$; see Proposition \ref{prop_U}. As a byproduct, we also show how to obtain a set of consistent boundary data satisfying the discrete weak KAM problem; see Proposition \ref{prop_gene}.

For step (ii), based on the globally adjusted boundary data and Peierls barriers $h(y;x_i)$, we obtain a local variational representation for $W(x)$, which only depends on the adjacent boundary data and barrier functions; see Proposition \ref{prop_local}. This local trimming procedure reduces the computations, as shown in the examples in Figure \ref{fig_w000}.
After explaining the variational construction for $W(x)$, in Proposition \ref{prop_W}, we prove $W(x)$ is a global viscosity solution to
\begin{equation}\label{Htt}
H(W'(x), x)=W'(x)(W'(x)-U'(x))=0, \quad x\in \bS^1
\end{equation}
satisfying the boundary data uniquely determined via \eqref{wi}.

Section \ref{sec4} focuses on the weak KAM interpretation for $W^*(x)=W(x)-\min W(x)$. We characterize that the projected Aubry set $\mathcal{A}$ is equal to the projected Mather set $\mathcal{M}_0$ and is equal to all the critical points of $U(x)$. In Corollary \ref{cor_kam}, we prove $W^*(x)$ is a weak KAM solution to \eqref{Htt} of negative type, in which the calibrated curves tracking back to the projected Mather set and those curves are simply solved by the `uphill/downhill' first order ODEs; see \eqref{ODE2} and \eqref{ODE1}. Moreover, the constructed $W(x)$ is the maximal Lipschitz continuous viscosity solution satisfying the boundary data $W(x_i)=W_i$ given in \eqref{wi}. These boundary data are chosen via \eqref{wi} and Lemma \ref{lem:disW}, so that $W^*(x)=W(x)-\min_{x\in\bS^1} W(x)$ is the rate function for the large deviation principle of the invariant measures of the diffusion process on $\bS^1$. While all the invariant sets characterized above are the uniqueness sets for the weak KAM solutions to HJE \eqref{Htt}, there are other uniqueness sets and we show that the uniqueness sets must contain all the local maximums/minimums; see Lemma \ref{prop_mather}. After all these understandings from the weak KAM perspective, we give a probability interpretation for the weak KAM solution $W^*(x)$. 

{\blue In Section \ref{sec4_nonunique}, we provide more understandings of the obtained weak KAM solution $W^*(x)$, including   the exchange of  double limits and how one selects a meaningful weak KAM solution that captures the asymptotic behavior of the original stochastic process at each local attractors. In Proposition \ref{prop_exchange}, using the property that $W^*(x)$ is an invariant solution to the Lax-Oleinik semigroup representation for the corresponding dynamic HJE,  we prove that for a special initial distribution, the large time limit and the zero noise limit can be exchanged for the distribution $\rho_\eps(x,t)$ of the diffusion process, i.e., ``$\lim_{\eps \to 0} \bbs{-\eps \log \lim_{t\to +\8} \rho_\eps(x,t) }=\lim_{t\to +\8} \lim_{\eps \to 0} - \eps \log \rho_\eps(x,t).$''. In general, the double limits in both sides exist for any initial data. The RHS limit exists \cite[Theorem 2.1]{BSBS} but is not unique. However, the LHS limit is unique, which provides a selection principle. That is to say we first take the long time limit $\lim_{t\to +\8} \rho_\eps(x,t)$ which is unique due to ergodicity  and then take the zero noise limit $\eps \to 0$ due to the large deviation principle for invariant measures.
In Section \ref{sec5.3}, we discuss  our selection principle for weak KAM solutions, which are in general not unique; see examples in Section \ref{sec5.1}. $W^*(x)=\lim_{\eps \to 0} \bbs{-\eps \log \lim_{t\to +\8} \rho_\eps(x,t) }$ with the variational formula for boundary data $W(x_i)$ serves as a meaningful selection principle because it is proved to be the rate function of the large deviation principle for the invariant measures.   Indeed,
the associated viscous HJE computed from the WKB reformulation of the invariant measure $\pi_\eps(x) = e^{-\frac{W_\eps(x)}{\eps}}$ of the irreversible diffusion process is
\begin{equation}\label{1.5}
 W'_\eps(W'_\eps - U')=\eps (W_\eps-U)''.
\end{equation}
As the rate function of the associated irreversible diffusion process on $\bS^1$, $W^*(x)$ is unique and can be regarded as the limit of $W_\eps(x)$, in the sense of the large deviation principle (see \eqref{LD2}). \eqref{1.5} also provides a special construction of a viscosity approximation, which has a uniform vanishing viscosity limit. We point out that in general, the vanishing viscosity approximation method for stationary HJEs only has converged subsequences whose limit is  not unique. Our selection principle is different from the widely studied selection principle via the discount limit of the associated optimal control problem in an infinite time horizon. The discount limit method usually can not capture the long time behavior of the original dynamics.}
Based on the selection principle in the large deviation sense, the periodic Lipschitz continuous global energy landscape $W^*(x)$, determines the most probable states/path for the original stochastic dynamics as the noise goes to zero; see the generalized Boltzmann analysis through the calibrated curves in Section \ref{sec4_prob}. 

The remaining parts of this paper are organized as follows. In Section \ref{sec2}, we introduce the Langevin dynamics on the circle $\bS^1$ and describe the large deviation principle for the invariant measures with both an illustrative example and the abstract result by the Freidlin-Wentzell theory. In Section \ref{sec3}, we give explicit properties for the Peierls barriers and use them to prove the variational formula for the global adjustment of boundary data and to construct the global energy landscape $W^*(x)$. The local variational representation, the consistency check, and the viscosity solution property for $W^*(x)$ are given in Section \ref{sec3_localR}, Section \ref{sec3_U}, and Section \ref{sec3_vis}, respectively. Section \ref{sec4} focuses on the characterization of Aubry/Mather sets and the weak KAM solution properties. $W^*(x)$ being the weak KAM solution is proved in Section \ref{sec4_kam}, whose uniqueness depending on the self-consistent boundary data is discussed in Section \ref{sec4_gene}. The nonuniqueness of weak KAM solutions and our selection principle, compared with other selection methods, are given by Section \ref{sec4_nonunique}. The probability interpretation for Freidlin-Wentzell's construction of $W^*(x)$ is discussed in Section \ref{sec4_prob}.

\section{The rate function of large deviation principle for the invariant measure of Langevin dynamics on a circle $\bS^1$}\label{sec2}
We first introduce a very simple stochastic model, which however, contains all the representative properties to study the relations between the large deviation principle for invariant measures and the weak KAM theory. This is a one-dimensional irreversible diffusion process on the periodic domain $\bS^1$, in which the WKB reformulation for the invariant measure $\pi_\eps(x)$ gives a stationary HJE. In Section \ref{sec2_LD}, we first use a single-well non-periodic potential $U(x)$ to illustrate the local trimming of the potential $U(x)$ brought by the large deviation principle, and then we describe the general large deviation principle for the invariant measure $\pi_\eps(x)$, which was proved by Freidlin-Wentzell \cite{venttsel1969limiting, ventsel1970small}. The associated variational formula for the rate function will be introduced in detail in Section \ref{sec3}.

\subsection{Langevin dynamics on a circle $\bS^1$}
In this subsection, we first introduce a Langevin dynamics on the simplest closed manifold $\bS^1$.
We start from a Langevin dynamics on a circle $\bS^1$ with a drift in gradient form, i.e., there exists a smooth skew periodic potential $U(x)$ such that
$b(x) = -U'(x), x \in \bS^1$. This Langevin dynamics on $\bS^1$ reads
\begin{equation}\label{LV}
\ud x = - U'(x) \ud t + \sqrt{2\eps} \ud B.
\end{equation}
Here the skew periodicity of the smooth function $U(x)$ implies there exists a smooth periodic function $\tilde{U}(x)$ such that $U(x)=\tilde{U}(x)-\bar{b}x$ for a constant $\bar{b}$.
Therefore, $U'(x), x\in \bS^1$ is a smooth periodic function and
\begin{equation}\label{sk}
\int_0^1 U'(x) \ud x = U(1) - U(0) = -\bar{b}.
\end{equation}
We refer to the dashed black line in Figure \ref{fig_p_m} for an example of a skew periodic potential $U(x)$ with three local minimums in one skew period.

The  Kolmogolov forward equation corresponding to \eqref{LV} is
 \begin{equation}\label{FPt}
\pt_t \rho_\eps =  (\rho_\eps  U')' + \eps \rho_\eps'' \quad \text{ in } \bS^1.
\end{equation}
 Plugging  the WKB reformulation $\rho_\eps(x,t)= e^{-\frac{\psi_\eps(x,t)}{\eps}}$ into \eqref{FPt}  and then taking $\eps\to 0$, we obtain the dynamic HJE
  \begin{equation}
 \pt_t \psi(x,t) + H(\pt_x \psi(x,t), x) = 0, \quad x\in \bS^1,
 \end{equation}
where the Hamiltonian $H: \bR\times \bS^1 \to \bR$ is
 \begin{equation}\label{Hs}
 H(p,x) = p(p- U').
 \end{equation}
 Then the corresponding Lagrangian,  as the convex conjugate of $H(p,x)$, is given by
 \begin{equation}\label{LLL}
 \begin{aligned}
 L(s,x) = \sup_{p\in\bR} \bbs{s p - H(p,x)} = s p^* - H(p^*, x) = \frac{1}{4}(s+U'(x))^2,
 \end{aligned}
 \end{equation}
 where $p^*$ solves
$
 s = \pt_p H(p^*, x) = 2p^* -  U'(x).
$
It is easy to see  Hamiltonian $H(p,x)$ is strictly convex w.r.t. $p$, periodic w.r.t. $x$ while Lagrangian $L(s,x)$ is strictly convex w.r.t. $s$, periodic w.r.t. $x$. Another important property is 
\begin{equation}\label{L0}
 L(s,x)\geq 0  \text{ and }  L=0 \text{ if and only if } s=-U'(x).
 \end{equation}
 The above ODE flow $\dot{x}= - U'(x)$ can be naturally embedded into the Euler-Lagrangian flow $(x, \dot{x})(t)$ on the tangent bundle $T\bS^1.$ This special Lagrangian graph  $(x, -U'(x))$ enables explicit computations for invariant measures and action minimizing measures/curves; see \textsc{Ma\~n\'e} \cite{mane92}. So this  Lagrangian \eqref{LLL} is also known as the Ma\~n\'e Lagrangian \cite{figalli2012aubry}.

\subsection{The invariant measure $\pi_\eps(x)$ and the large deviation principle as $\eps\to 0$}\label{sec2_LD}
The corresponding invariant measure $\pi_{\eps}(x)$ satisfies the stationary Fokker-Planck equation
\begin{equation}\label{FP}
 \eps \pi_{\eps}''  + ( U' \pi_{\eps} )' = 0 \quad  \text{ in } \bS^1.
\end{equation}
Without loss of generality, we assume $\min U(x) = 0.$
The unique periodic positive solution $\pi_{\eps}$ is given by  
\begin{equation}\label{pie}
\pi_{\eps}(x)= C_\eps e^{-\frac{U(x)}{\eps}} \int_x^{x+1} e^{\frac{U(y)}{\eps}} \ud y, \quad x\in \bS^1, 
\end{equation}
where $C_\eps$ is a normalization constant such that $\int \pi_\eps = 1.$ 
The integral function in \eqref{pie} can be regarded as a corrector to make $\pi_{\eps}(x)$ to be periodic.
Indeed, recast \eqref{pie} as 
$
\pi_{\eps}(x)\propto  \int_x^{x+1} e^{\frac{U(y)-U(x)}{\eps}} \ud y,
$
which is periodic.

If $\bar{b}=0$, then $U(x)$ is periodic and the above integral in \eqref{pie} is a constant. Thus the Langevin dynamics \eqref{LV} is a reversible process and the periodic invariant measure is directly given by $\pi_{\eps}(x)\propto e^{-\frac{U(x)}{\eps}}$. Indeed, from \eqref{pie}, one can compute the steady flux 
\begin{equation}
J_\eps = \eps \pi_\eps' + U' \pi_\eps =  \eps C_\eps  \bbs{ e^{-\frac
{\bar{b}}{\eps} } -1 }.
\end{equation}
$\bar{b}=0$ is equivalent to $J_\eps = 0$ pointwise and thus equivalent to reversibility of the Langevin dynamics \eqref{LV}.
 Then it is obvious that $U(x) = - \eps \log \pi_{\eps}(x)$ is the rate function in the large deviation principle for the reversible invariant measure $\pi_{\eps}(x)$. 

However, if $\bar{b}\neq 0$, then the Langevin dynamics \eqref{LV} is irreversible and the invariant measure does not have a straightforward formula to serve as a rate function in the large deviation principle.
In this case, we define a WKB reformulation
\begin{equation}\label{We}
W_{\eps}(x) := -\eps \log \pi_{\eps}(x) = U(x) - \eps \log  \int_x^{x+1} e^{\frac{U(y)}{\eps}} \ud y - \eps \log C_\eps, \quad x\in \bS^1.
\end{equation}
From \eqref{pie}, since the solution $\pi_\eps(x)$ to \eqref{pie} has a unique closed formula, so $W_\eps(x) = -\eps \log \pi_\eps(x)$ can be uniquely computed upto a constant.
 
{\blue If as $\eps \to 0$, the limit $W_\eps(x) \to W^*(x)$ exists for some periodic function $W^*(x)$, then this limit $W^*(x) $ is the rate function for the large deviation principle of the invariant measure $\pi_{\eps}(x)$. For a peculiar case that $U(x)$ is strictly monotone, then $\pi_\eps(x)= \frac{C}{|U'(x)|} + O(\eps)$ does not have an exponential asymptotic behavior. In this case, $W^*(x)\equiv 0$. Hence we only consider the case when $U(x)$ has minimums.}

\subsubsection{Illustration of the Laplace principle for a single-well potential $U(x)$}
In this subsection, we use the following simple example with a single well non-periodic potential $U(x)$ to explicitly compute and simulate the convergence from $W_\eps(x)$ to the globally defined, periodic, Lipschitz continuous rate function $W^*(x);$ see plots of $U(x),W_\eps(x)$ and $W(x)$ in Figure \ref{fig_ld}.

Take 
\begin{equation}\label{U1}
U(x) = \cos(2\pi x)-\cos(\pi x) + \frac{9}{8}, \quad x\in[0,1].
\end{equation}
Then $U(x)$, $x\in(0,1)$ is a  single basin of attractor of the stable state $x_0=\frac{1}{\pi}\arccos \frac14$, $U_{\min}=U(x_0)=0$ and the boundary difference is  $-\bar{b}=U(1) - U(0) = 2.$ One can do skew periodic extension to a $C^1$ function on $\bR$ by $U(x+k)=U(x)-\bar{b}k$, $k\in\bZ$. Define $x^* = \frac{2 }{3}$ which has the same value as $U_{\text{exit}}:=U(0)=U(x^*)$ for the exit problem, we have 
$$0=U_{\min}<U_{\text{exit}} = U(0)=U(x^*)<U(1)=U_{\max}.$$
Since $U_{\max}=U(1)$, $b+U_{\max}=U(0)$ and \eqref{pie} can be reformulated with a different $C_\eps$
\begin{equation}\label{int1}
\begin{aligned}
\pi_\eps(x)=& C_\eps e^{-\frac{U(x)}{\eps}} \bbs{\int_x^1 e^{\frac{U(z)-U_{\max}}{\eps}} \ud z +   \int_0^x e^{\frac{U(z)-\bar{b} - U_{\max}}{\eps}} \ud z}\\
=&C_\eps e^{\frac{-U(x)}{\eps}} \bbs{\int_0^x e^{\frac{U(z)-U(0)}{\eps}} \ud z + \int_x^1 e^{\frac{U(z) - U(1)}{\eps}} \ud z}, \quad x\in \bS^1.
\end{aligned}
\end{equation}
Since $\min U(x)=0$, $e^{\frac{-U(x)}{\eps}}=O(1)$ and by the Laplace principle
  $\int_0^x e^{\frac{U(z)-U(0)}{\eps}} \ud z + \int_x^1 e^{\frac{U(z) - U(1)}{\eps}} \ud z \geq O(1)$ as $\eps \to 0$. Thus $C_\eps \leq O(1)$.
When $x\in[0,x^*]$, $U(x)=\min\{U(x), U_{\text{exit}}\}$, then one  can directly apply the Laplace principle for the integrals in \eqref{int1}. But for $x\in[x^*, 1]$, $U(0)=\min\{U(x), U_{\text{exit}}\}$, so the first integration in \eqref{int1} shall be recast as
\begin{equation}
e^{\frac{-U(x)}{\eps}} \int_0^x e^{\frac{U(z)-U(0)}{\eps}} \ud z = e^{\frac{ - U(0)}{\eps}}  \int_0^x e^{\frac{U(z)-U(x)}{\eps}} \ud z.
\end{equation}
Then from WKB reformulation \eqref{We}, we obtain
\begin{equation}\label{Lap0}
\begin{aligned}
W_\eps(x) = -\eps \log \pi_\eps(x)& = \min\{U(x), U_{\text{exit}}\} \\
&-\eps \log \left\{ \begin{array}{cc}
 C_\eps\bbs{\int_0^x e^{\frac{U(z)-U(0)}{\eps}} \ud z + \int_x^1 e^{\frac{U(z) - U(1)}{\eps}} \ud z}, & x\in [0,x^*];\\
C_\eps\bbs{\int_0^x e^{\frac{U(z)-U(x)}{\eps}} \ud z + e^{\frac{  U(0)-U(x)}{\eps}}\int_x^1 e^{\frac{U(z) - U(1)}{\eps}} \ud z}, & x\in [x^*,1].
\end{array} \right.
\end{aligned}
\end{equation}
Now every integral term in the above desingularization formula is $O(1)$ and can be numerically implemented.
Then by the Laplace principle, we show the last term in \eqref{Lap0} is a $o(1)$ term
\begin{equation}
\begin{aligned}
 \eps \log \left\{ \begin{array}{cc}
  C_\eps\bbs{\frac12\sqrt{\frac{2\pi\eps}{|U''(0)|}} + \frac12\sqrt{\frac{2\pi\eps}{|U''(1)|}}}, & x\in [0,x^*];\\
 C_\eps\bbs{\frac{\eps}{U'(x)} + e^{\frac{  U(0)-U(x)}{\eps}}\frac12\sqrt{\frac{2\pi\eps}{|U''(1)|}}  } , & x\in [x^*,1].
\end{array} \right.
\end{aligned}
\end{equation}
Hence, we obtain the rate function $W^*(x)$ for the large deviation principle
\begin{equation}\label{converge}
W_{\eps}(x) \to W^*(x) :=   \min\{U(x),U_{\text{exit}}\}, \quad \text{ in } \bS^1. 
\end{equation}
In Figure \ref{fig_ld}, the WKB reformulation $W_\eps(x)$  is plotted with $\eps=0.05,0.01, 0.005, 0.003, 0.002,0.001$. A  uniform convergence from $W_\eps(x)$ to the rate function $W^*(x)$ is shown as $\eps \to 0$.
 
 \begin{figure}
 \includegraphics[scale=0.5]{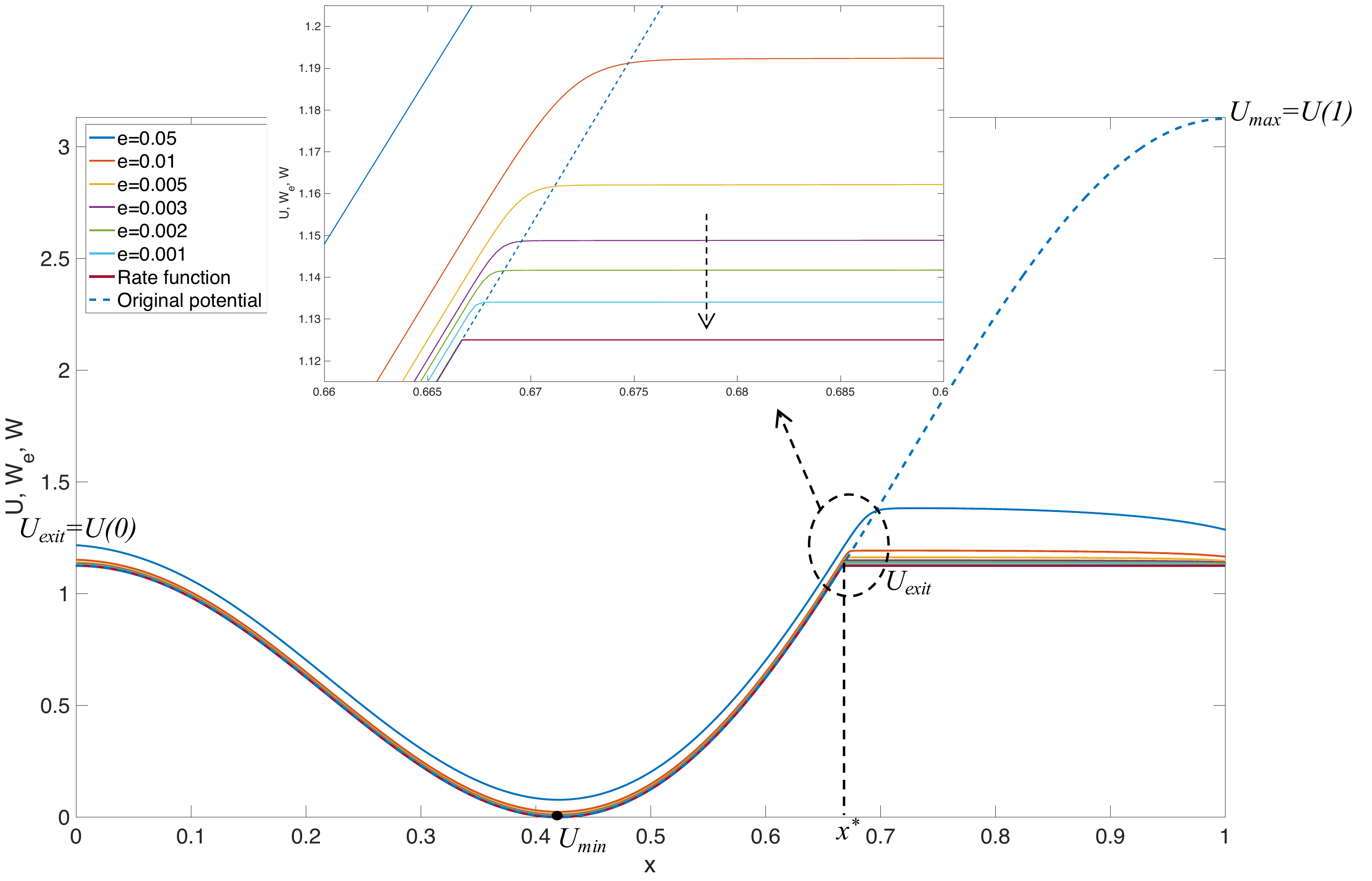} 
  \caption{The plot of $U(x)$ in \eqref{U1},  $W_\eps(x)$ in \eqref{Lap0} with different values of $\eps$ and the rate function $W^*(x)$ in \eqref{converge}. As $\eps\to 0$, a uniform convergence from $W_\eps(x)$ to the rate function $W^*(x)$ is shown with a zoom-in plot near the non-differential connection point, where $U(x)$ is cut off from above by the constant $U_{\text{exit}}$. $U(x)$ is skew periodic while both $W_\eps(x)$ and $W(x)$ are periodic.}\label{fig_ld}
 \end{figure}
 
For other cases where $U(x)$ has multiple wells, this simple cut-off (local trimming) by a constant as described in the above example is not enough. A globally defined adjustment for the values at each local minimum of $U(x)$ needs to be constructed first and then apply the local trimming procedures; see Section \ref{sec3_FW}. Finding the correct global energy landscape with multiple wells, after proper global adjustment and gluing, is important for the rare events computations; cf. \cite{Ebook}, \cite{ge2012landscapes}, \cite{zl16}, and \cite{GLLL20}. The global energy landscape correctly measures the action/energy required for the transition from one state to another state, which in general is not the original potential function $U(x)$, as we have already seen from the above simple example (Figure \ref{fig_ld}).

For the general case, the explicit variational formula for the limit $\lim_{\eps\to 0} W_\eps(x) = W^*(x) = W(x) - \min W(x)$ was discovered by \cite[Section 8]{venttsel1969limiting} (see below \eqref{FW_W}) and we will describe it in detail in Section \ref{sec3}.
In \cite[Section 8]{venttsel1969limiting} (see also \cite[Chapter 6, Theorem 4.3]{FW}), 
this limit $W^*(x)$ is proved to be the rate function for the large deviation principle of $\pi_\eps(x)$ in the following sense. For any $\gamma>0$, there exists $\delta_0$ such that for any $\delta < \delta_0$, there exists $\eps_0$ such that for any $\eps< \eps_0 $
\begin{equation}\label{LD}
W(x) - \min W(x) - \gamma \leq -\eps \log \pi_\eps(B_\delta(x)) \leq W(x) - \min W(x) + \gamma.
\end{equation}
This statement is equivalent to Varadhan's   definition \cite[Definition 2.2]{VR} for the large deviation principle on compact domain. Indeed, 
taking $\liminf$ and $\limsup$ w.r.t. $\delta$ and $\eps$, we have
\begin{equation}\label{LD1}
\begin{aligned}
W(x) - \min W(x) - \gamma \leq& \liminf_{\delta \to 0} \liminf_{\eps \to 0} -\eps \log \pi_\eps(B_\delta(x)) \\
 \leq& \limsup_{\delta \to 0} \limsup_{\eps \to 0} -\eps \log \pi_\eps(B_\delta(x))  \leq W(x) - \min W(x) + \gamma.
 \end{aligned}
\end{equation}
Thus since $\gamma$ is arbitrary, 
\begin{equation}\label{LD2}
\begin{aligned}
W(x) - \min W(x) =& \liminf_{\delta \to 0} \liminf_{\eps \to 0} -\eps \log \pi_\eps(B_\delta(x)) 
=& \limsup_{\delta \to 0} \limsup_{\eps \to 0} -\eps \log \pi_\eps(B_\delta(x)). 
 \end{aligned}
\end{equation}
Thus, \eqref{LD1} is exactly the lower bound and upper bound estimates in \cite[Definition 2.2]{VR} and thus implies the large deviation principle for the invariant measure $\pi_\eps(x)$ with the rate function $W(x)-\min W(x)$. In one dimension, \cite{FG2012} provides a direct proof for the limit of $-\eps \log \pi_\eps(x)$ via Laplace's principle and recovers the variational formula \eqref{FW_W}.

We remark that for general stochastic processes,  for instance for the large deviation principle in the thermodynamic limit for the invariant measures of the chemical reaction process \cite{GL22t}, one can also directly study the 
  upper semicontinuous (USC) envelope of the WKB reformulation $W_\eps(x)$
  \begin{equation}
  \overline{W}(x) := \limsup_{\eps \to 0, \, x_j \to x} W_\eps(x_j)
  \end{equation}
 and the
  lower semicontinuous (LSC)  envelope of $W_\eps(x)$
  \begin{equation}
  \underline{W}(x) := \liminf_{\eps \to 0, \, x_j \to x} W_\eps(x_j).
  \end{equation}
  Then by these definitions, if the large deviation principle \eqref{LD2} holds then necessarily one obtain
\begin{equation}
\underline{W}(x) \leq W(x)-\min W(x) \leq \overline{W}(x).
\end{equation}
In \cite[Proposition 4.1]{GL22t}, under the detailed balance condition for the chemical reaction process, the USC envelope $\overline{W}(x)$ is proved to be a USC viscosity solution to the corresponding stationary HJE in the Barron-Jensen sense \cite{Barron_Jensen_1990}.

  \subsubsection{Variational formula for $W^*(x)$ through Varadhan's lemma}
  Recall the WKB reformulation $W_\eps(x)$  of invariant measure $\pi_\eps(x)$. From the large deviation principle \eqref{LD2} for the invariant measure $\pi_\eps(x)$,  
  Varadhan's lemma \cite{Varadhan_1966, VR} provide another variational formula for the rate function $W^*(x)$. Below, we carry out details for this formula.

  Using Varadhan's lemma  \cite[Theorem 2.5]{VR}, we know for any test function $f\in C(\bS^1)$,
  \begin{equation}
  \sup_{y\in \bS^1} (f(y) - W^*(y) ) = \lim_{\eps \to 0} -\eps \log \int_{ \bS^1 } e^{ \frac{ f(y) }{\eps}  } \pi_\eps( y ) \ud y.
  \end{equation}
  Denote the integral above as $A_\eps$ and compute it via the closed formula \eqref{We}
  \begin{equation}
  \begin{aligned}
  A_\eps := \int_{\bS^1} e^{ \frac{ f(y) }{\eps}  } \pi_\eps( y ) \ud y =& C_\eps \int_0^1 e^{\frac{f(y) - U(y)}{\eps}}  \int_y^{y+1} e^{\frac{U(z)}{\eps} } \ud z  \ud y\\
  =& C_\eps \int_0^1 \int_0^z e^{\frac{f(y) - U(y) + U(z)}{\eps}} \ud y \ud z + C_\eps e^{-\frac{\bar{b}}{\eps}} \int_0^1 \int_z^1 e^{\frac{f(y) - U(y) + U(z)}{\eps}} \ud y \ud z.
  \end{aligned}
  \end{equation}
  Here we used exchange order of integrals and skew periodicity \eqref{sk}.
    However, there is no simple way to directly study this globally defined limiting function. Thus we go back to Freidlin-Wentzell's variational construction below.
  
In the next section,   under the assumption that there are finite critical points $x_i$ for $U(x)$, we reinterpret and give an alternative proof for the construction of  a global periodic energy landscape $W(x)$ from locally defined quasi-potential after adjusting levels and then  proper trimming and gluing.   
 From the formula of $W_\eps(x)$ in \eqref{We}, the limiting rate function $W^*(x)$ can be viewed as the original potential $U(x)$ with an additional corrector computed from the Laplace principle for $\eps \log  \int_x^{x+1} e^{\frac{U(y)}{\eps}} \ud y$. This provides the recipes to (i) globally adjust the levels $W(x_i)$ of each quasi-potential for each stable basin via \eqref{wi}  and then (ii) to locally trimming from above and glue to construct $W^*(x)=W(x)-\min_i W(x_i)$ via \eqref{FW_W} or \eqref{localW}.
  This global energy landscape is continuous and is proved to be the 
 rate function of the large deviation principle for the invariant measure of Langevin dynamics on a circle $\bS^1$ \cite[Chapter 6, Theorem 4.3]{FW}.  We will also prove that $W^*(x)$ is a viscosity solution to HJE \eqref{HJE_w}; see Proposition \ref{prop_W}.
 
 \section{Freidlin-Wentzell's variational construction of periodic Lipschitz continuous global energy landscape $W^*(x)$}\label{sec3}
In this section, we focus on the detailed description of Freidlin-Wentzell's variational construction of the rate function $W^*(x)$; see \eqref{disW} and \eqref{FW_W}. Using the one-dimensional irreversible diffusion on torus \eqref{FPt}, we give an alternative elementary proof for the variational formula \eqref{disW} for determining the boundary data and elaborating explicit properties of $W^*(x)$ such as the shape of non-differentiable part. Those boundary values are globally defined and are the most crucial ingredient to obtain the unique, Lipschitz continuous, periodic global energy landscape that can correctly represent the exponentially small probability in the large deviation principle.
Before giving the variational formula, we revisit and prove the detailed characterizations of two essential concepts of least action functions: the quasi-potential (aka. the Ma\~n\'e potential) and the Peierls barrier. Based on these explicit properties, we then give the construction of  a global energy landscape in Section \ref{sec3_FW} based on (i) the global adjustment for boundary data on the local minimums and (ii) a local trimming procedure via adjacent boundary data and the Peierls barrier; see Lemma \ref{lem:disW} and Proposition \ref{prop_local}. At the last, we give a consistent verification to show the variational representation of $W^*(x)$ is indeed   reduced to the original potential function $U(x)$ if the diffusion process is reversible; see Proposition \ref{prop_U}.
   
Let us first clarify that we always work on the case that $U(x)$ has finite many critical points indexed as follows.  Assume there are $k$ stable local minimums
 \begin{equation}\label{id1}
x_1, x_2, \cdots, x_k
 \end{equation}
 of $U(x)$, interleaved by $k$ unstable local maximums
  \begin{equation}\label{id2}
x_{\frac12}, x_{1+\frac12}, \cdots, x_{k+\frac12}=x_{\frac12}+1.
 \end{equation}
 With out loss of generality, we assume
 \begin{equation}
 0 =  x_{\frac12} < x_1 < x_{1+\frac12}< x_2 < \cdots < x_k < x_{k+\frac12} = 1.
 \end{equation}
 Denote other duplicated points outside $[0,1]$ as $x_{i+\ell k} = x_{i}+\ell\in \bR$ for any $\ell\in\bZ.$
 
 \subsection{Properties of Peierls barrier $h(y; x_i)$ and Ma\~n\'e potential $v(y;x_i)$}
In this subsection, we revisit two essential concepts of least action functions: the quasi-potential  (aka. the Ma\~n\'e potential) and the Peierls barrier. In our one dimensional example, we further explore the explicit shape characterizations for those least action functions, which are important properties for the later construction of global energy landscape.

 \subsubsection{Quasi-potentials and critical Ma\~n\'e value $c^*=0$}\label{sec3_critical}
Quasi-potential is an essential concept introduced in the Freidlin-Wentzell theory, while the local quasi-potential within a basin of attraction of stable states is widely used  for computing the barrier of  exit problems for stochastic processes. We explain below the globally defined quasi-potential, which is now called  the  Ma\~n\'e potential following the convention in the weak KAM theory.

Given any starting point $x_0$, not necessarily critical points, the  Ma\~n\'e potential is defined as 
\begin{equation}\label{value}
 \begin{aligned}
 &v(y; \,  x_0):=\inf_{T\geq 0,\vg(0)= x_0,\,\, \vg(T)= y} \int_0^T \bbs{L(\dot{\vg}(t),\vg(t)) + c}\ud t.
 \end{aligned}
 \end{equation}
It is well known that the critical Ma\~n\'e value for the Lagrangian \eqref{LLL} is zero $c^*=0$. So from now on, we drop $c$ in the definition of the Ma\~n\'e potential.  As it is an essential concept, we provide  descriptions of four  characterizations for $c^*$ below.  
 
 (i) The  definition of the critical Ma\~n\'e value is, cf.  \cite{contreras1999global}
\begin{equation}\label{mane}
c^*= \sup\{c\in \bR; \, \exists \text{ closed curve }  x(\cdot) \text{ s.t. } \int_0^T (L(\dot{x}(t),x(t))+c) \ud t<0 \}.
\end{equation}
Since $L\geq 0$, so we know at least  $c^*\leq 0$. On the other hand,  if $c^*<0$, then one can choose a standing curve $ x(t)\equiv x_i$ at a steady state $x_i$  such that $\dot{x} = -U'(x)\equiv 0$.
 Then one have $L\equiv 0$ while $\int_0^T (L(\dot{x}(t),x(t))+c^*) \ud t<0$. Thus $c^*=0.$

(ii) From \cite[Definition 4.2.6]{fathi2008weak}, one can verify
\begin{equation}
c^*:= \inf\{c\in\bR; \, \text{ there exists } u\in \text{Lip}(\bS^1) \text{ s.t. }H(u'(x),x)\leq c \text{ a.e. }\}=0.
\end{equation}
 Indeed, on the one hand, taking $u\equiv 0$ implies $H(u'(x),x)=0$, so at least $c^*\leq 0$. On the other hand, 
since $H = (u' - \frac{U'}{2})^2 - (\frac{U'}{2})^2$, for any $u\in \text{Lip}(\bS^1)$, $H\geq 0$ at critical point $x_i$ of $U$. Thus $c^*$ can not be negative, so $c^*=0.$
 
  There are another two methods for characterizing $c^*:$ 
(iii) $c^*$ can also be  computed via the min-max problem, cf.   \cite[Theorem 4.1]{evans2008weak}: 
 $$c^*= - \inf_{\varphi\in C^1(\bS^1)} \max_{x\in \bS^1} \varphi'(x)(\varphi'(x) - U'(x)) = 0.$$
 (iv)  $c^*$ can be computed via  action minimizing (Mather) measures, cf.   \cite[Theorem 2.7]{evans2008weak}: Let $\mathcal{P}_{\text{inv}}( \bR \times \bS^1)$ be the collection of  probability measures $\mu$   on $\bR \times \bS^1$ that is invariant under the Lagrangian flow.  Then
 $$c^* = - \inf_{\mu \in \mathcal{P}_{\text{inv}} }  \int_{\bR \times \bS^1} L(s,x) \ud \mu(s,x)=0.$$
 See also \cite{mane1996generic} for a relaxed minimization which relaxes the invariant Lagrangian flow condition.
 The  measure achieving the minimum is called Mather measure. There are many Mather measures $\mu$   on $\bR \times \bS^1$ for our problem. For instance, we take $\mu = \delta(s)\delta(x-x_i)$, where $x_i$ is a steady state, and it is easy to verify the minimum $c^*=0$ is achieved.

\subsubsection{Characterization of Peierls barriers $h(y; x_i)$ on $\bS^1$}
 We point out that for the above case that the starting point $x_i$ is a stable/unstable critical point of $U(x)$,   another
 important concept, called
 the Peierls barrier is defined as
\begin{equation}\label{pb}
h(y; x_i) := \liminf_{T\to  +\8} \inf_{\vg(0)= x_i,\,\, \vg(T)= y} \int_0^T L(\dot{\vg}(t),\vg(t))\ud t.
\end{equation}
 
From the computations for left/right Ma\~n\'e values  in \eqref{quasi1} and \eqref{quasi2}, it is easy to see for $x_i$ being a critical point of $U(x)$, then  $v(y; x_i ) = h(y; x_i)$ for any $y\in \bS^1$. Thus from now on, we use   Peierls barrier $h(y; x_i)$ instead of $v(y; x_i)$ whenever the starting point is a critical point of $U(x)$. 

 Before we explain explicitly the global energy landscape $W^*(x)$ construction, we characterize the explicit formula for Peierls barriers $h(y; x_i)$ with the specific non-differential point, connection shape, and periodicity; see Figure \ref{fig_p_m}. This will also serve as a key observation for justifying the weak KAM solution later. In this paper, we always assume the orientation of $x\in \bS^1$ belongs to an interval $x\in (a,b)$ is counterclockwise.
   \begin{prop}\label{prop_h}
 Assume there are $k$ stable local minimums   of $U(x)$, interleaved by $k$ unstable local maximums indexed as \eqref{id1} and \eqref{id2}.   Then
 \begin{enumerate}[(i)]
\item The Peierls barriers $h(y; x_i)\geq 0$ is Lipschitz continuous and periodic.
\item There exists $x^*\in \bS^1$ such that $h(y;x_i)$ is   noninceasing in $(x^*, x_i)$ to zero and then nondecreasing in $(x_i, x^*+1)$ back to the same level $h(x^*; x_i)=h(x^*+1; x_i)$.
\item The only one possible non-differential point is the connection point $x^*$, where either an increasing function connected to a constant  or a constant connected to a decreasing function. That is to say, $h(y;x_i)$ is a $C^1$ function cut off at most once  by a constant from above. 
\item  $h(y; x_i)$ is the maximal Lipschitz continuous  viscosity solution to HJE
\begin{equation}\label{HJE_h}
 H(h'(y),y) = h'( h'-U' )=0, \quad y\in \bS^1
\end{equation}
and satisfies $h(x_i; x_i)=0.$
\end{enumerate}
   \end{prop}
   This proposition on the characterization of Peierls barriers is basically known in the weak KAM theory but here we give detailed properties on the periodicity and explicit shape of $h(y;x_i)$.
   \begin{proof}

  First,    we define a  right barrier function for $y\in [x_i, x_{i+k}]\subset \bR$ 
  \begin{equation}
 h_R(y; x_i):= \inf_{T\geq 0,\vg(0)=x_i,\,\, \vg(T)=y} \int_0^T\frac14| \dot{\vg} + U'(\vg)|^2\ud t
  \end{equation}
   for an  exit problem starting from $x_i$ to the right until the point $x_{i+k}$ passing through several local minimums $x_{i+1}, \cdots, x_{i+k-1}\in \bR$. To see explicitly the formula for the barrier function, from each stable minimums $x_i$ to its adjacent critical points, we can first compute  
\begin{equation}\label{quasi1}
 \begin{aligned}
 h_R(x_{i+ \frac12}; \, x_i)=&\inf_{T\geq 0,\vg(0)=x_i,\,\, \vg(T)=x_{i+\frac12}} \int_0^T\frac14| \dot{\vg} + U'(\vg)|^2\ud t \\
 =&\inf_{T\geq 0,\vg(0)=x_i,\,\, \vg(T)=x_{i+\frac12}}\int_0^T\bbs{\frac14| \dot{\vg} - U'(\vg)|^2 + \dot{\gamma} U' } \ud t \geq  U(x_{i+\frac12}) - U(x_i).
 \end{aligned}
 \end{equation}
 Here the equality holds if and only if $ \dot{\vg} =U'(\vg)$ and $\gamma(+\8)=x_{i+ \frac12}$, so $h_R(x_{i+ \frac12}; \, x_i) =U(x_{i+\frac12}) - U(x_i)$. It is usually referred as the `uphill' path from $x_i$ to $x_{i+\frac12}$; cf.  \cite{FW}. Similarly, the left barrier from $x_i$ to the left to $x_{i-\frac12}$ is
 \begin{equation}
 h_L(x_{i-\frac12}; x_i):= \inf_{T\geq 0,\vg(0)=x_i,\,\, \vg(T)=x_{i-\frac12}} \int_0^T\frac14| \dot{\vg} + U'(\vg)|^2\ud t =  U(x_{i-\frac12}) - U(x_i).
\end{equation} 
Conversely, for    the `downhill' path starting from $x_{i+\frac12}$ along $ \dot{\vg} =-U'(\vg)$ and $\gamma(+\8)=x_{i+1}$, we have
 \begin{equation}\label{quasi2}
 \begin{aligned}
 h_R( x_{i+1}; x_{i+ \frac12})=&\inf_{T\geq 0,\vg(0)=x_{i+\frac12},\,\, \vg(T)=x_{i+1}} \int_0^T\frac14| \dot{\vg} + U'(\vg)|^2\ud t = 0.
 \end{aligned}
 \end{equation}
 Thus 
\begin{equation}\label{quasi}
 \begin{aligned}
 h_R(x_{i+ 1}; \, x_i)=h_R(x_{i+ \frac12}; \, x_i)= U(x_{i+\frac12}) - U(x_i);\\
 h_L(x_{i- 1}; \, x_i)=h_L(x_{i- \frac12}; \, x_i)= U(x_{i-\frac12}) - U(x_i).
 \end{aligned}
 \end{equation}
 Other barriers passing through multiple wells can be computed  repeatedly.
     
     Thus   the barrier formula $h_R$ for this one dimensional least action problem (multiple exit  problems) is given by   a  least action problem for piecewisely $C^1$ curve connecting $x_i$ to $y\in [x_i, x_{i+k}]$
\begin{equation}\label{hr}
\begin{aligned}
h_R(y; x_i) = \left\{ \begin{array}{ccc}
U(y) -  U(x_i), & y\in [x_{i}, \, x_{i+\frac12}], &\text{increase}; \\
U(x_{i+\frac12}) - U(x_i), & y \in [x_{i+\frac12}, x_{i+1}], &\text{constant};\\
U(x_{i+\frac12}) - U(x_i) + U(y)-U(x_{i+1}), & y\in [x_{i+1}, \, x_{i+\frac32}], &\text{increase};\\
\cdots\\
\sum_{j=i}^{i+k-1} [ U(x_{j+\frac12}) - U(x_j) ], & y\in [x_{i+k-\frac12}, \, x_{i+k}], &\text{constant}.
\end{array}
 \right.
\end{aligned}
\end{equation}  
We emphasize $U(x)$ is skew periodic function defined on the whole $\bR$, so $h_R(y; x_i)$ is well-defined.  
It's easy to see $h_R(y; x_i)$ is a nondecreasing $C^1([x_i, x_{i+k}])$ function.  


Similarly, a nonincreasing $C^1([x_{i-k}, x_i])$ function $h_L(y; x_i), y\in [x_{i-k}, x_i]$ can be computed to serve as a   left  barrier function for the  exit problem starting from $x_i$ to the left until the point $x_{i-k}$
 \begin{equation}\label{hl}
\begin{aligned}
h_L(y; x_i) = \left\{ \begin{array}{ccc}
\sum_{j=i-k+1}^{i} [ U(x_{j-\frac12}) - U(x_j) ], & y\in [x_{i-k }, \, x_{i-k+\frac12}], &\text{constant};\\
\cdots\\
U(x_{i-\frac12}) - U(x_i) + U(y)-U(x_{i-1}), & y\in [x_{i-\frac32}, \, x_{i-1}], &\text{decrease};\\
U(x_{i-\frac12}) - U(x_i), & y \in [x_{i-1}, x_{i-\frac12}], &\text{constant};\\
 U(y) -  U(x_i), & y\in [x_{i-\frac12}, \, x_{i}], &\text{decrease}.
\end{array}
 \right.
\end{aligned}
\end{equation} 
For other points, due to screw periodicity of $U(x)$, one can naturally define
$$h_L(y\pm 1,x_{i\pm k}) = h_L(y, x_i),\quad h_R(y\pm 1, x_{i\pm k}) = h_R(y,x_{i}).$$
See Figure \ref{fig_hLR} for the illustration of the left barrier $h_L(y;x_i)$ and the right barrier $h_R(y;x_i).$ 
 \begin{figure}
 \includegraphics[scale=0.36]{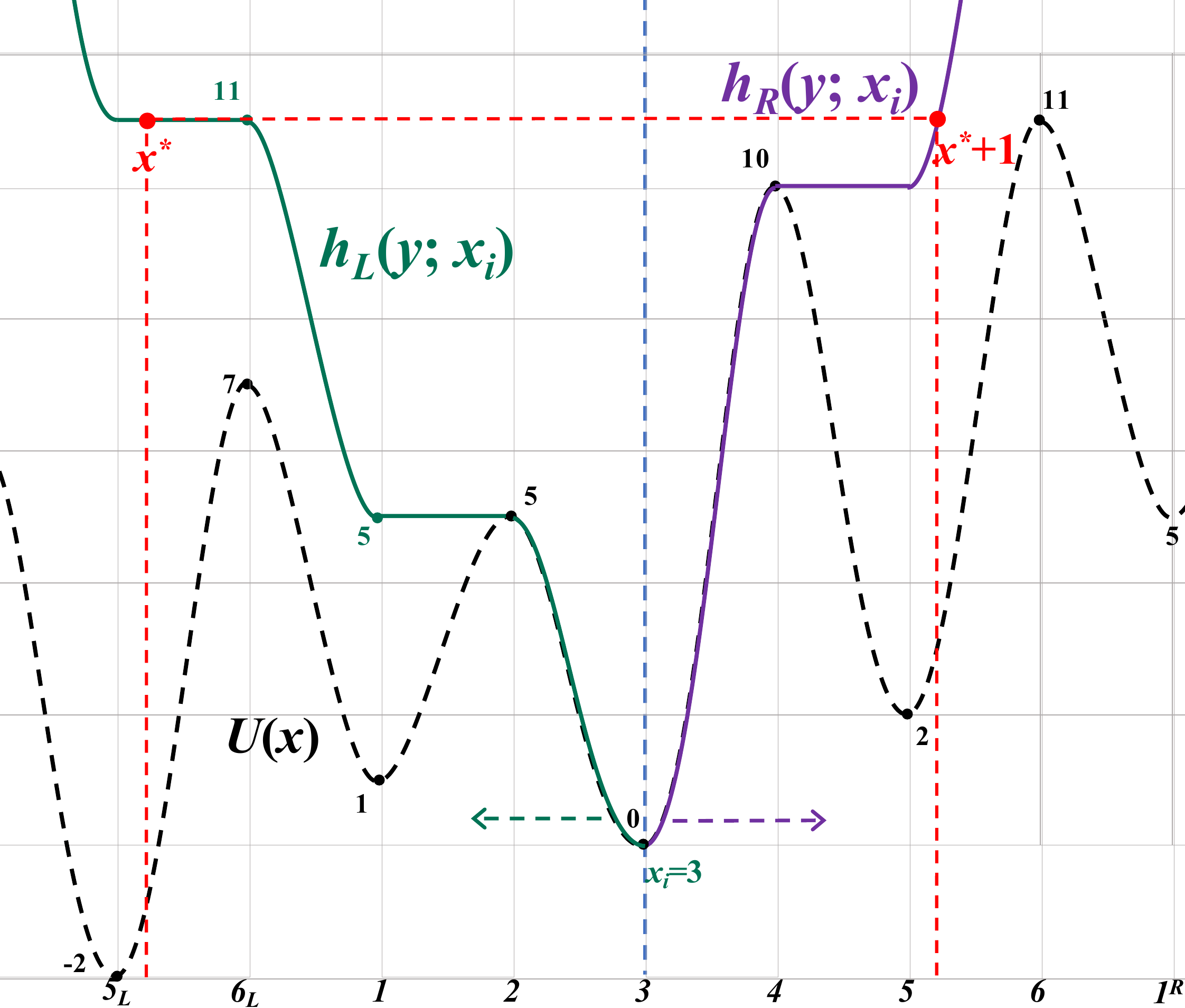}
 \caption{The construction of the   Peierls barrier $h(y;x_i)$ via the left barrier $h_L(y;x_i)$ and right barrier $h_R(y;x_i)$.  The  dashed black line is the skew periodic potential $U(x)$ with three local minimums $1,0,2$ in $\bS^1=[6_L,6]$. The solid green  line is the left barrier starting from $x_i=3$ and is nondecreasing clockwise.  The solid purple  line is the right barrier starting from $x_i=3$ and is nondecreasing counterclockwise. After finding $x^*$ such that $h_L(x^*; x_i) = h_R(x^*+1; x_i)$, $h(y;x_i)$ is   periodic, Lipschitz continuous with only one cut off by a constant from the above at the non-differential point $x^*$.  } \label{fig_hLR}
 \end{figure}
   
  Second, since $U'(x)$ is periodic, for $y\in \bS^1$, we compute the Peierls barrier
  \begin{equation}\label{hmin}
  h(y; x_i):= \inf_{T\geq 0,\vg(0)=x_i,\,\, \vg(T)=y } \int_0^T\frac14| \dot{\vg} + U'(\vg)|^2\ud t.
  \end{equation}
  Using $h_R, h_L$, we represent $h(y; x_i)$, $y\in[0,1]$ as
  \begin{equation}\label{hlrmin}
  h(y; x_i) = \left\{ 
\begin{array}{cc}
\min \{ h_L(y; x_i), \, h_R(y+1; x_{i})\},\, &y\leq x_i,\\
\min \{ h_L(y-1; x_{i}), \, h_R(y; x_i)\},\, &y>x_i,
\end{array}   
  \right.
  \end{equation}
 which will be proved below to be a periodic, Lipschitz  continuous function.
   We first give a key oberservation, which will be used in the characterization of the shape of   the global energy landscape as well.   Notice $[x_{i}, \, x_{i+\frac12}] = [x_{i-k }, \, x_{i-k+\frac12}]+1$ is on the increasing interval of the $i$-th well $U(x)$, where $h_R$ is increasing and $h_L$ is a constant. Similarly, $[x_{i-\frac12}, \, x_{i}] = [x_{i+k-\frac12}, \, x_{i+k}]-1$ is on the decreasing interval of the $i$-th well $U(x)$, where $h_R$ is a constant and $h_L$ is decreasing. 
Below, we proceed to characterize $h(y;x_i)$.
  
  Step 1.  Since $h_R(y;x_i)$ is nondecreasing and $h_L(y;x_i)$ is nonincreasing, there always exists 
$x^*$ such that $h_L(x^*; x_i) = h_R(x^*+1; x_i)$.
 Therefore, for $x^*\leq y \leq x_i$, the minimum \eqref{hlrmin} is attained at $h_L(x)$ while for $x_i\leq y \leq x^*+1$, the minimum \eqref{hlrmin} is attained at $h_R(x)$. Thus the Peierls barriers is given by
  \begin{equation}\label{pblr}
  h(y; x_i)=\left\{ \begin{array}{cc}
   h_L(y; x_i), &  x^*\leq y \leq x_i,\\
   h_R(y; x_i), & x_i\leq y \leq x^*+1.
   \end{array}
   \right.
  \end{equation}
Immediate consequences are that $h(y;x_i)$ is $C^1$ function in $(x^*, x^*+1)$, nonincreasing in $(x^*, x_i)$ to zero, and nondecreasing in $(x_i, x^*+1)$ back to the same level  
$$h(x^*;x_i) = h(x^* + 1 ; x_i).$$
 Thus $h(y;x_i)$ has continuously periodic extension and the only possible non-differential point is $x^*+ \bZ$; see Figure \ref{fig_hLR} for the construction of $h(y;x_i)$ via $h_L(y;x_i), h_R(y;x_i).$ Hence we obtained the conclusion (i) and (ii).
  
  Step 2. We prove the type of  the  non-differentiablity for  point $x^*$.
  
  First, there exists $i-k\leq \ell\leq i$ such that $x^* \in [x_{\ell-\frac12}, x_{\ell+\frac12}]$. We only need to consider three cases. Case (1), $x^* = x_{\ell}$ or $x_{\ell\pm\frac12}$, then $h'(x^*;x_i) = h'_L(x^*)=h'_R(x^*)=0$ is differentiable. Case (2), if $x^* \in (x_{\ell}, x_{\ell+\frac12})$, then  from the formula \eqref{pblr}, we know $h'(x^*_+; x_i) = h_L'(x^*_+; x_i)=0$ while $h'(x^*_-; x_i) = h_R'(x^*_-; x_i)=U'(x^*)>0$. This case implies that an increasing function is connected to a constant at the non-differential point $x^*$. Case (3), if $x^* \in (x_{\ell-\frac12}, x_{\ell})$, then the left derivative $h'(x^*_-; x_i)=0$  while the right derivative   exists and is negative $h'(x^*_+; x_i)=U'(x^*)<0$. This implies that a constant is connected to a decreasing function at the non-differential point $x^*$. Therefore, we obtained the conclusion (iii).

 Step 3. One can directly verify $h(y;x_i)$ is a viscosity solution based on the   definition cf. \cite[Page 5]{bardi1997optimal}.
The maximality of $h(y;x_i)$ follows \cite[Theorem 2.41]{Tran21} or \cite[Theorem 2.4]{fathi2009hausdorff} only with small modifications. 
Let $\tilde{u}(y)$ be a Lipschitz continuous viscosity subsolution to \eqref{HJE_h} satisfying $\tilde{u}(x_i)=0$ and thus   an almost everywhere subsolution satisfying $H( \nabla\tilde{u}(x), x) \leq 0$ a.e. $x\in \bS^1$. So for any absolutely continuous curve $\gamma(\cdot)$ with $\vg(0) = x_i$ and $\vg(t)=y\in \bS^1$, we have
\begin{equation*}
\begin{aligned}
\tilde{u}(y) - \tilde{u}(x_i) =& \int_0^t \nabla \tilde{u}(\vg(s)) \cdot \dot{\vg}(s) \ud s\\
 \leq &  \int_0^t \bbs{L(\dot{\vg}(s), \vg(s)) + H(\nabla \tilde{u}(\vg(s)),\vg(s))  } \ud s \leq \int_0^t  L(\dot{\vg}(s), \vg(s)) \ud s.
 \end{aligned}
\end{equation*}
Then taking infimum w.r.t. $\vg$ and $\liminf$ w.r.t. $t$, we obtain
\begin{equation}
\tilde{u}(y)  \leq  \liminf_{t\to +\8} \inf_{\vg; \vg(0)=x_i, \vg(t)=y} \int_0^t  L(\dot{\vg}(s), \vg(s)) \ud s = h(y; x_i).
\end{equation}
Thus   the  Peierls barrier $h(y;x_i)$    is   the maximal Lipschitz continuous viscosity solution satisfying \eqref{HJE_h}.
   \end{proof}
   
   \begin{rem}\label{rem_h0}
   The shape of the  Peierls barrier $h(y;x_{i+\frac12})$ starting from the local maximums $x_{i+\frac12}$ can be characterized with the same arguments. The only difference is $h(y;x_{i+\frac12})=0$ for $y\in[x_i, x_{i+1}].$ Then outside $[x_i, x_{i+1}]$, one can use $h_R(y; x_{i+1})$ and $h_L(y; x_i)$ to construct $h(y;x_{i+\frac12}).$
   \end{rem}

 \subsection{Freidlin-Wentzell's variational construction for the rate function $W(x)$ via boundary values $W_i$ at stable states and Peierls barriers $h(y;x_i)$}\label{sec3_FW}
 Based on the previous explicit characterization of Peierls barriers starting from each stable states, in this subsection, we describe and give an alternative proof for  Freidlin-Wentzell's variational formula for determining the boundary values. Those boundary values are globally defined and are the most crucial ingredient to obtain the unique, Lipschitz continuous, periodic global energy landscape that can correctly represent the exponentially small probability in the large deviation principle. After obtaining the global adjustment of boundary values, the variational construction for the rate function  $W^*(x)$ is indeed a local trimming procedure; see Section \ref{sec3_localR} for the local representation of $W(x)$.
   
 \subsubsection{Determine boundary values $W_j$ on stable states}
 Now we determine the boundary values $W(x_i)$ at stable minimum $x_i$.
 For any $j=1,\cdots, k$, recall $h_R(x_i; x_{j+1})$ defined in \eqref{hr}.  To compute a counterclockwise path connecting 
 $x_{j+1}\in \bS^1$ to $x_i\in \bS^1$, we introduce a tilde notation for the total cost of  this   path on $\bS^1$ 
 \begin{equation}
\tilde{h}_R(x_i; x_{j+1}):= \left\{
\begin{array}{cc}
 h_R(x_i; x_{j+1}), & j < i,\\
 h_R(x_{i+k}; x_{j+1}), & j\geq i.
\end{array}
\right. 
 \end{equation} 
Similarly, using $h_L(x_i; x_{j})$ defined in \eqref{hl},    the total cost for a clockwise  path connecting 
 $x_{j}\in \bS^1$ to $x_i\in\bS^1$ is
 \begin{equation}
 \tilde{h}_L(x_i; x_{j}) := \left\{
\begin{array}{cc}
 h_L(x_i; x_{j}), & j \geq i,\\
 h_L(x_{i-k}; x_{j}), & j< i.
\end{array}
\right.
 \end{equation}

 Then following \cite[Chapter 6, eq. (4.2)]{FW},  define
\begin{equation}\label{wi}
W_i := \min_{j=1,\cdots,k} (\tilde{h}_R(x_i; x_{j+1}) \, + \,\tilde{h}_L(x_i; x_{j})), \quad i=1, \cdots, k.
 \end{equation} 
 We refer to the example in Figure \ref{fig_w000} (left) for a globally adjusted boundary data $W_i$, $i=1,2,3$ satisfying \eqref{wi} and the construction of $W^*(x)$ based on those boundary data. With these specially adjusted boundary data, $W^*(x)$ is proved to be the rate function of the large deviation principle for invariant measures \eqref{LD}, \cite[Chapter 6, Theorem 4.3]{FW}. We also provide a coarse grained Markov chain interpretation  in Appendix \ref{app2}.
 
 Notice all $x_i$ are stable critical points, so   the explicit formula for the Peierls barrier $h(x_i; x_j)$ in \eqref{hlrmin} is recast as
 \begin{equation}
 h(x_i; x_j)=\min \{ \tilde{h}_R(x_i; x_{j}) \, , \,\tilde{h}_L(x_i; x_{j}) \}, \quad i, j=1, \cdots, k.
 \end{equation}
 In the following lemma, we prove the boundary data satisfying the variational formula \eqref{wi} is indeed a consistent data set satisfying the discrete weak KAM problem \eqref{disW}.

 \begin{lem}\label{lem:disW}
 Let $h(x_i; x_j)$ be the Peierls barrier. 
 The values of $W_i$, $i=1, \cdots, k$ defined in \eqref{wi} solves the discrete weak KAM problem
 \begin{equation}\label{disW}
 W_i = \min_{j=1,\cdots, k} \{ W_j + h(x_i; x_j)\}, \quad \forall i = 1, \cdots, k.
 \end{equation}

 \end{lem}
 \begin{proof}
To verify \eqref{disW}, it is sufficient to verify for any $\ell$, there exists $m$ such that
\begin{equation} \label{4.7}
  W_i\leq  \tilde{h}_R(x_i; x_{m+1}) \, + \,\tilde{h}_L(x_i; x_{m}) \leq \tilde{h}_R(x_j; x_{\ell+1}) \, + \,\tilde{h}_L(x_j; x_{\ell}) + h(x_i, x_j).
 \end{equation}
 Indeed, taking minimum in $\ell$, we have $W_i \leq W_j  + h(x_i, x_j)$
  and then taking minimum in $j$, we have $W_i\leq \min_{j=1,\cdots, k} \{ W_j + h(x_i; x_j)\}.$ Particularly, the equality holds for $j=i$.
 
 Now we prove \eqref{4.7} for the case $h(x_i; x_j)=\min \{ \tilde{h}_R(x_i; x_{j}) \, , \,\tilde{h}_L(x_i; x_{j}) \} = \tilde{h}_L(x_i; x_{j})$, and the other one has the same argument.\\
 (i) If $\ell \in \{ j, j+1, \cdots, i - 1 \}$ is on the counterclockwise path from $x_j$ to $x_i$, then taking $m = \ell$, we obtain 
 \begin{equation} 
\tilde{h}_R(x_i; x_{\ell+1}) \, + \,\tilde{h}_L(x_i; x_{\ell}) \leq \tilde{h}_R(x_j; x_{\ell+1}) \, + \,\tilde{h}_L(x_j; x_{\ell}) \,+ \tilde{h}_L(x_i; x_{j})
 \end{equation}
 due to $\tilde{h}_R(x_i; x_{\ell+1}) \leq \tilde{h}_R(x_j; x_{\ell+1}) $ and $\tilde{h}_L(x_i; x_{\ell}) \leq \tilde{h}_L(x_j; x_{\ell}) + \tilde{h}_L(x_i; x_{j}).$
 \\(ii) If $\ell \in \{ i, i+1, \cdots, j-1\}$ is on the clockwise path from $x_j$ to $x_i$, then taking $ m = i -1$, since $\tilde{h}_R(x_i; x_{i})=0$ and $ \tilde{h}_L(x_i; x_{i-1}) \leq  \tilde{h}_L(x_j; x_{\ell}) + \tilde{h}_L(x_i; x_{j}),$ we have
  \begin{equation} 
\tilde{h}_R(x_i; x_{i}) \, + \,\tilde{h}_L(x_i; x_{i-1}) \leq \tilde{h}_R(x_j; x_{\ell+1}) \, + \,\tilde{h}_L(x_j; x_{\ell}) \,+ \tilde{h}_L(x_i; x_{j}).  
 \end{equation}
 Thus \eqref{4.7} is proved, so does the lemma.
 \end{proof}
 
  \subsubsection{Variational construction for $W^*(x)$ via boundary values $W_i$ on stable states $x_i$}
 With the above boundary values $W_i$, $i=1, \cdots, k$ on all the stable minima, the global energy landscape is defined as \cite[Chapter 6, Theorem 4.3]{FW}
\begin{equation}\label{FW_W}
 W(x) = \min_{j=1,\cdots, k} \{ W_j + h(x; x_j)\}, \quad \forall x\in \bS^1.
 \end{equation}
 Later in Section \ref{sec4}, we will prove that $W(x)$ is indeed a weak KAM solution to the HJE \eqref{HJE_w}. We also characterize the corresponding projected Aubry set $\mathcal{A}$ in Section \ref{sec4_Aubry}. After including the induced boundary values on   other critical point (local maximums) in $\mathcal{A}$, this $W(x)$ satisfies the usual representation (cf. \cite[Theorem 7.4]{Tran21}) via the boundary data on the projected Aubry set for the weak KAM solution; see Lemma \ref{lem:Wmm}.

 We remark the boundary values $W_i$ to the discrete weak KAM problem \eqref{disW} are not uniquely determined because $W_i=0,\, i=1, \cdots, k$ are also admissible boundary values satisfying \eqref{disW}; see Figure \ref{fig_w000} (right) for instance. Meanwhile, a constant shift of $W_i$ is also a solution to \eqref{disW}. We refer to Section \ref{sec4_nonunique} for examples of non-uniqueness.
 
 However, the  construction described above using the uniquely determined boundary data $W_i$ and the trimming of $W_i+h(x;x_i)$ has clear probability meaning via the large deviation principle for the invariant measure $\pi_\eps(x)$. From \cite[Chapter 6, Theorem 4.3]{FW},
  \begin{equation}\label{rate}
 W^*(x) = W(x)-\min_i W(x_i)
  \end{equation}
   gives the rate function in the large deviation principle for the invariant measure $\pi_\eps(x)$ to the Langevin dynamics \eqref{LV} on $\bS^1.$ 
   In Section \ref{sec4}, we will explore more properties of $W(x)$ from the weak KAM viewpoint and use the corresponding projected Aubry/Mather set to give a probability interpretation of the global energy landscape $W^*(x)$.

  \subsubsection{Local representation for $W(x)$}\label{sec3_localR}
  Based on the globally adjusted boundary values $W_j$, the rate function $W(x)$ can be constructed in \eqref{FW_W}. In the following proposition, we show that the variational formula \eqref{FW_W}   indeed has a local representation depending only on the boundary values of the adjacent local minima and barrier functions. This procedure is thus also referred as a local trimming procedure. We refer to Figure \ref{fig_w000} for an illustration of the local trimming.

  \begin{prop}\label{prop_local}
  Let $W(x)$ be given by \eqref{FW_W} with boundary values $W_j$, $j=1, \cdots, k.$ Assume the boundary values $W_j$ satisfy the discrete weak KAM problem \eqref{disW}. Then $W(x)$ has a local representation that, for any $x\in [x_i, x_{i+1}]$ for some $i=1, \cdots,k$
  \begin{equation}\label{localW}
  W(x) = \min\{ W_i + h_R(x;x_i) , \, W_{i+1} + h_L(x; x_{i+1}) \},
  \end{equation}
  where $h_R(x; x_i)$ and $h_L(x; x_{i+1})$ is the locally defined, right/left barrier  functions in \eqref{hr} and \eqref{hl}. Consequently, at each local maximums $x_{i+\frac12}$, there is at most one flat connection either on the left of $x_{i+\frac12}$ or on the right of $x_{i+\frac12}.$
  \end{prop}
  \begin{proof}
  Assume $W(x)$ defined in \eqref{FW_W} is achieved at $j$, i.e.,
  $$W(x) = \min_{j=1,\cdots, k} \{ W_j + h(x; x_j)\} =    W_j + h(x; x_j) . $$
  
  Case (1), if $h(x;x_j)$ is achieved via clockwise path, then
  \begin{equation}
  W_j + h(x; x_j) = W_j + h(x_{i+1}, x_i) + h_L(x, x_{i+1}) \geq W_j  + h_L(x, x_{i+1}).
\end{equation}  
Case (2), if $h(x; x_j)$ is achieved via counterclockwise path, then
  \begin{equation}
  W_j + h(x; x_j) = W_j + h(x_{i}, x_i) + h_L(x, x_{i}) \geq W_j  + h_R(x, x_{i}).
\end{equation}

Therefore, combining both cases, we have
\begin{equation}
W_j + h(x; x_j) \geq \min\{W_i + h_R(x;x_i) , \, W_{i+1} + h_L(x; x_{i+1})\}. 
\end{equation}
From \eqref{pblr}, we further know
\begin{equation}
\begin{aligned}
W(x)=W_j + h(x; x_j) & \geq \min\{W_i + h_R(x;x_i) , \, W_{i+1} + h_L(x; x_{i+1})\} \\
&\geq \min\{W_i + h(x;x_i) , \, W_{i+1} + h(x; x_{i+1})\} \geq \min_{j=1,\cdots, k} \{ W_j + h(x; x_j)\}=W(x).
\end{aligned}
\end{equation}
This gives \eqref{localW}.

At last, notice in $(x_i, x_{i+\frac12})$, $h_R(x;x_i)$ is strictly increasing while $h_L(x;x_{i+1})$ is constant; likewise in $(x_{i+\frac12}, x_{i})$ with $h_L(x;x_{i+1})$ being strictly decreasing and $h_R(x;x_i)$ being a constant. Thus the curves meet at most once and hence there is at most one flat connection either on the left of $x_{i+\frac12}$ or on the right of $x_{i+\frac12}.$
This completes the proof. 
\end{proof}   
   
 We remark this explicit local representation is not only helpful for constructing the global energy landscape $W(x)$ (see Figure \ref{fig_w000}) but also enables us to identify whether the least action curves track backward in time or defined both forward and backward for the whole time $t\in \bR.$
  
  \subsection{Consistency check for Freidlin-Wentzell's variational formulas when $U(x)$ is periodic }\label{sec3_U}
  If the original potential $U(x)$ is periodic with $\min U(x)=0$, i.e., $\bar{b}=0$ and $\pi_\eps(x)$ is a reversible invariant measure
  given by $\pi_\eps(x) = C_\eps e^{-\frac{U(x)}{\eps}}$. Since $\min U(x)=0$, by Laplace's principle
  $\int e^{-\frac{U(x)}{\eps}} \sim O(1)$ as $\eps \to 0$. Thus $C_\eps \sim O(1)$. Following Varadhan's equivalent definition \cite[Definition 2.2]{VR} for the large deviation principle on compact domain, we compute
  \begin{equation}\label{LD-U}
\begin{aligned}
U(x) \leq & \liminf_{\delta \to 0} \liminf_{\eps \to 0} -\eps \log \pi_\eps(B_\delta(x)) 
\leq & \limsup_{\delta \to 0} \limsup_{\eps \to 0} -\eps \log \pi_\eps(B_\delta(x)) \leq U(x).
 \end{aligned}
\end{equation}
  Thus $U(x)$ is the rate function for the large deviation principle of invariant measure $\pi_\eps(x).$
  
As a consistent check, we prove below that, the constructed global energy landscape $W^*(x)$ from \eqref{FW_W} and \eqref{rate}, is exactly the original potential $U(x)$.
  
    \begin{prop}\label{prop_U}
    Let $W^*(x)$ be constructed from \eqref{FW_W} and \eqref{rate} with boundary data \eqref{disW}.
  If the potential $U(x)$, $x\in \bS^1$ is periodic with $\min U(x)=0$, then  $W^*(x) = U(x)$.
  \end{prop}
  \begin{proof}
  Step 1. We prove Freidlin-Wentzell's variational formula \eqref{wi} $W_i = U(x_i) + \text{ const}.$
  
  First, define the total right/left barrier in one period as 
  \begin{equation}
  J_{\pm} = \sum_{\ell=1}^k \bbs{ U_{\ell\pm \frac12}  - U_\ell }. 
  \end{equation}

  Second, for any $x_{i}, x_j \in \bS^1$, by elementary calculations, we see
  \begin{equation}
  \tilde{h}_R(x_{i+1};x_{j+1}) - \tilde{h}_R(x_{i};x_{j+1}) = \left\{ \begin{array}{cc}
  U(x_{i+\frac12}) - U(x_i), & j\neq i;\\
U(x_{i+\frac12}) - U(x_i) - J_+, & j=i.
  \end{array}
   \right.
  \end{equation}
  Similarly, we have
    \begin{equation}
  \tilde{h}_L(x_{i+1};x_{j}) - \tilde{h}_L(x_{i};x_{j}) = \left\{ \begin{array}{cc}
  -\bbs{U(x_{i+\frac12}) - U(x_{i+1})}, & j\neq i;\\
-\bbs{U(x_{i+\frac12}) - U(x_{i+1})} +  J_-, & j=i.
  \end{array}
   \right.
  \end{equation}
  Thus we obtain
      \begin{equation}
  \tilde{h}_R(x_{i+1};x_{j+1}) - \tilde{h}_R(x_{i};x_{j+1}) + \bbs{  \tilde{h}_L(x_{i+1};x_{j}) - \tilde{h}_L(x_{i};x_{j})} = \left\{ \begin{array}{cc}
 U(x_{i+1}) - U(x_{i}), & j\neq i;\\
U(x_{i+1}) - U(x_{i}) +  J_- - J_+, & j=i.
  \end{array}
   \right.
  \end{equation}
  
  Third, when $U(x)$ is periodic, it is easy to verify $J_+ - J_- =0$. Therefore,
  \begin{equation}
  \tilde{h}_R(x_{i+1};x_{j+1}) + \tilde{h}_L(x_{i+1};x_{j})    =   \tilde{h}_R(x_{i};x_{j+1}) +\tilde{h}_L(x_{i};x_{j})  +U(x_{i+1}) - U(x_{i}).
   \end{equation}
   Taking minimum w.r.t. $j$ and using the definition of $W_i$ in \eqref{wi}, we have
   \begin{equation} 
\begin{aligned}
W_{i+1}  = W _{i}  +U(x_{i+1}) - U(x_{i}).
 \end{aligned}
\end{equation}
This implies
\begin{equation} 
\begin{aligned}
W_i -U(x_i) = \text{ const }, \quad i =1, \cdots, k.
 \end{aligned}
\end{equation}

Step 2. We prove 
\begin{equation}\label{tm_h}
U(x) = \min_{i=1,\cdots, k} \bbs{U(x_i) + h(x;x_i)}.
\end{equation}

Fix any $x\in [i-\frac12, i+\frac12]$, on the one hand, from the definition of quasi-potential, we have
\begin{equation}
U(x) = U(x_i) + h(x; x_i).
\end{equation}
On the other hand, we need to prove for any $j\neq i$
\begin{equation}
U(x) \leq U(x_j) + h(x; x_j).
\end{equation}
From the property of $h(x; x_j)$ in Proposition \ref{prop_h}, below we only prove case that $x$ belongs to the nonincreasing part   of $h(x; x_j)$. Another case that $x$ belongs to the nondecreasing part of $h(x; x_j)$ has the same argument.

Since $x$ belongs to the nonincreasing part   of $h(x; x_j)$ and $U(x)$ is periodic, so a clockwise path from $x_j$ to $x$ can be regarded as $x<x_{i+\frac12}<x_{i}< \cdots < x_{j-\frac12}<x_j$. Thus
\begin{equation} 
\begin{aligned}
U(x) \leq&  U(x_{i+1})+h(x;x_{i+1}) \\
\leq&     U(x_{i+2})+ h(x_{i+1};x_{i+2}) + h(x; x_{i+1}) = U(x_{i+2}) + h(x;x_{i+2}) \\
\leq& \cdots \leq U(x_j) + h(x; x_j).
 \end{aligned}
\end{equation}
Thus, we obtain \eqref{tm_h}. Replace $U(x_i)$ by $U(x_i)+c = W_i$ in \eqref{tm_h}, and then 
\begin{equation}
U(x) + c = \min_{i=1,\cdots, k} \bbs{W_i + h(x;x_i)} = W(x).
\end{equation}

  \end{proof}

 \subsection{The global energy landscape $W^*(x)$ is a viscosity solution}\label{sec3_vis}
 
 Recall Hamiltonian \eqref{Hs}, we now prove the continuous periodic global landscape $W(x)$ constructed in \eqref{FW_W} is a viscosity solution to
the stationary HJE in $\bS^1$.
 
  \begin{prop}\label{prop_W}
 Assume there are $k$ stable local minima   of $U(x)$, interleaved by $k$ unstable local maxima indexed as \eqref{id1} and \eqref{id2}. Let $W(x)$ be constructed in \eqref{FW_W}.  Then
 \begin{enumerate}[(i)]
\item  $W(x)$ is Lipschitz continuous and periodic.
\item There are at most one  non-differential point belonging to each increasing (resp. decreasing) interval of the original potential $U(x)$, where $W(x)$ is  an increasing function  connected to a constant (resp.   a constant connected to a decreasing function). Particularly, $W(x)$ is differentiable at all the critical points $x_i, x_{i+\frac12}$, $i=1, \cdots, k$.
\item  $W(x)$ is a    viscosity solution to HJE
\begin{equation}\label{HJE_w}
 H(W'(x),x) = W'( W'-U' )=0, \quad x\in \bS^1
\end{equation}
and satisfies the boundary data $W(x_j)=W_j$ at $x_j$, $j=1, \cdots, k.$
\end{enumerate}
   \end{prop}
\begin{proof}
First, from Proposition \ref{prop_h}, $h(y;x_i)$ is Lipschitz continuous and periodic, so by the definition in \eqref{FW_W}, $W(x)$ satisfies conclusion (i).

Second, similar to the observations for $h_R(x;x_i)$ defined in \eqref{hr}, we characterize the shape of $W(x)$. For each increasing interval of $U(x)$,  $h(x; x_i) + W_i$  with different $x_i$, have only three possible shapes: constant, increasing part of $U(x)+\text{ const }$, or increasing function $U(x)+\text{ const }$  connected to a constant.   It is easy to verify that the minimum among all those $h(x; x_i)+ W_i$ gives $W(x)$ in this increasing interval, which remains to be one of these three types. Thus  there is  at most one non-differential connection point for $W(x)$, where an increasing function  is connected to a constant. The scenario for each decreasing interval of $U(x)$ is similar, where the only possible connection point is a constant connected to a decreasing function. This complete conclusion (ii).

Notice the number of non-differential points are finite and are of the same shape as in Proposition \ref{prop_h}, so the verification of the viscosity solution to \eqref{HJE_w} of conclusion (iii) is exactly the same as that of Proposition \ref{prop_h}.
For the boundary conditions, recall that from Lemma \ref{lem:disW}, the boundary data satisfies the discrete weak KAM problem \eqref{disW}. Thus $W(x_i) = \min_{j=1, \cdots, k}\{ W_j+h(x;x_j)\} = W_i$ for all $i=1, \cdots, k.$ This finishes the proof.
\end{proof}

\section{The global energy landscape $W^*(x)$ is a weak KAM solution}\label{sec4}

In this section, we first characterize the projected Aubry set $\mathcal{A}$  and uniqueness sets for the weak KAM solutions. Then in Proposition \ref{thm1} and Corollary \ref{cor_kam}, we  prove the main result that the global energy landscape $W(x)$ constructed in \eqref{FW_W} is a weak KAM solution, from which, each calibrated curve and the projected Mather set $\mathcal{M}_0$ can be determined. The projected Mather set $\mathcal{M}_0$ is indeed the  projected Aubry set $\mathcal{A}$, which are all the critical points of the original potential $U(x)$. Moreover, the constructed $W(x)$ is the maximal Lipschitz continuous viscosity solution satisfying the boundary data $W(x_i)=W_i$ given in \eqref{wi}. These boundary data is chosen via Lemma \ref{lem:disW} so that $W^*(x)=W(x)-\min_{x\in\bS^1} W(x)$ is the rate function for the large deviation principle of the invariant measures of the diffusion process on $\bS^1$. Hence this gives a meaningful selection principle for   weak KAM solutions to \eqref{HJE_w}. In subsections \ref{sec4_prob} and \ref{sec4_nonunique}, we give more discussions on the probability interpretations and different selection principles.

\subsection{Projected Aubry set $\mathcal{A}$ and the uniqueness set}\label{sec4_Aubry}
In this subsection, we first characterize the projected Aubry set $\mathcal{A}$, which is a uniqueness set for the weak KAM solution to HJE \eqref{HJE_w}. We also show  that all   of uniqueness sets must include all the local maxima/minima of $U(x)$ in Lemma \ref{prop_mather}. This includes the projected Mather set $\mathcal{M}_0$, which is also a uniqueness set for the weak KAM solution to HJE \eqref{HJE_w}. 
Second, we prove the variational formula of $W(x)$ defined in \eqref{FW_W} can also be represented   via the boundary data on the projected Aubry set $\mathcal{A}$. This representation, after extended to the projected Aubry set $\mathcal{A}$, is the usually variational representation for the weak KAM solution; cf. \cite[Theorem 7.4]{Tran21}.
{\blue We point out the results presented in this section are known from the general weak KAM theory. However, for our simple one-dimensional example, the proofs are explicit and simple. So we only outline  proofs for completeness. }

For the one dimensional Hamiltonian \eqref{Hs}, an equivalent characterization for the projected Aubry set $\mathcal{A}$  \eqref{AA} is given by using the 
viscosity solutions to HJE \eqref{HJE_w}. 
\begin{lem}
Assume there are $k$ stable local minimums   of $U(x)$, interleaved by $k$ unstable local maximums indexed as \eqref{id1} and \eqref{id2}.
The projected Aubry set for the Hamiltonian  \eqref{Hs} is
\begin{equation}\label{AA}
\mathcal{A} : = \{x_i,x_{i+\frac12}\,;\, i=1,\cdots, k\}.
\end{equation}
\end{lem}
\begin{proof}
From \cite[Definition 7.32]{Tran21}, the projected Aubry set $\mathcal{A}$ is all the starting points $x$ such as the Ma\~n\'e potential $v(y;x)$ is a viscosity solution to \eqref{HJE_w} on $\bS^1$.
Indeed, from Lemma \ref{lem:mane} (see also \cite[Theorem 2.41]{Tran21} or \cite[Theorem 2.4]{fathi2009hausdorff}), we know the Ma\~n\'e potential $v(y;x)$ for any $x\in\bS^1$ is the maximal Lipschitz continuous viscosity subsolution. On the one hand, Proposition \ref{prop_h} shows that if $x$ is a critical point, then  $v(y;x)=h(y; x)$ is a viscosity solution to \eqref{HJE_w}. Thus critical points belong to the projected Aubry set $\mathcal{A}$. On the other hand, for any $x$ not being the critical points, then from Lemma \ref{lem:mane}, the shape of $v(y;x)$  violates the viscosity supersolution test for HJE \eqref{HJE_w} and thus $\mathcal{A}$ must be a subset of all the critical points. This gives the characterization of the  projected Aubry set \eqref{AA}.
\end{proof}
We remark there are also other characterizations for the projected Aubry set $\mathcal{A}$, see for instance
 \cite{Contreras2001, fathi2008weak}
\begin{equation}\label{h00}
\mathcal{A}=\{ x\in \bS^1; \, h(x,x)=0\}.
\end{equation}
Notice the Lagrangian $L(s,x)\geq 0$ and satisfies the property \eqref{L0}. Thus \eqref{h00} can also be used to conclude the same characterization \eqref{AA}. 
 
Now we recall that the projected Mather set $\mathcal{M}_0$ \cite[Theorem 4.12.6]{fathi2008weak} is a uniqueness set for weak KAM solutions. Here a uniqueness set $M$ means that if two weak KAM solutions $u$ and $\tilde{u}$ coincide on an  $M$, then they must be same everywhere. In the following Lemma \ref{prop_mather}, we prove the projected Mather set $\mathcal{M}_0$ in our example must contain all the local maxima $\{x_{i+\frac12}; \, i=1, \cdots, k\}$ of $U(x).$ Later in Proposition \ref{thm1}, we will prove the projected Mather set is exactly the projected Aubry set $\mathcal{A}$.
{\blue \begin{lem}\label{prop_mather}
All the uniqueness sets $M\subset \mathcal{A}$ of the weak KAM solutions to HJE \eqref{HJE_w}  must contain all the local maxima $\{x_{i+\frac12}\,; \, i=1, \cdots, k \}$ and local minima $\{x_{i}\,; \, i=1, \cdots, k \}$ of $U(x).$ 
\end{lem}}

\begin{proof}
Let $M\subset \mathcal{A}$ be a uniqueness set for the weak KAM solutions to HJE \eqref{HJE_w}. We prove for any $i$, the maximum point $x_{i+\frac12} \in M.$ 

Using the argument by contradiction, we assume if $x_{i+\frac12} \notin M$ for some $i$. Then we can choose the boundary values $W(x^*)=0$ for all $x^*\in M$. It is easy to see zero boundary values always satisfy the  consistent condition for the discrete weak KAM problem on $M$, i.e.,
 \begin{equation}\label{disW1}
 W(x) = \min_{x^*\in M} \{ W(x^*) + h(x; x^*)\}, \quad \forall x \in M.
 \end{equation}
 Then using $W(x^*)$, $x^*\in M$ to construct a weak KAM solution $W(x)$, $x\in \bS^1$. Particularly, we have
 \begin{equation} 
 W(x_{i+\frac12}) = \min_{x^*\in M} \{ 0 + h(x_{i+\frac12}; x^*)\} >0.
 \end{equation}
On the other hand, setting $\tilde{W}(x_{i+\frac12}) = 0$, together with zero values in $M$, we can verify they still satisfy the consistent condition for the discrete weak KAM problem on the subset $M \cup \{x_{i+\frac12}\} \subset \mathcal{A}$, thus   we can construct another weak KAM solution
 \begin{equation}\label{disW2}
\tilde{W}(x) = \min_{x^*\in M \cup \{x_{i+\frac12}\}} \{ W(x^*) + h(x; x^*)\}, \quad \forall x \in \bS^1.
 \end{equation}
One can see $\tilde{W}(x^*)= W(x^*)=0$ for $x^*\in M$ but $\tilde{W} \neq W$ at $x_{i+\frac12}.$ This contradicts with the definition of the uniqueness set. Similar arguments apply to the local minimums. 
\end{proof}

\begin{rem}\label{rem_zero}
One can see an illustration of  the uniqueness set in   Figure \ref{fig_w000} (Right). Indeed, although $W(x_i)=0$, $i=1, \cdots,k$, $W(x)\not\equiv 0$ in  Figure \ref{fig_w000} (Right). But if adding additional boundary values $W(x_{i+\frac12})=0$ for $i=1, \cdots, k$, then the only solution is $W(x)\equiv 0.$ This is because $h(y;x_{i+\frac12})=0$ for $y\in[x_i, x_{i+1}]$; see  Remark \ref{rem_h0}. 
\end{rem}

Next, we observe $W(x)$ in \eqref{FW_W} is determined by the boundary values on the set of all the stable critical points of $U(x)$.  Then the variational construction \eqref{FW_W} induces  the values of $W(x)$   at all the unstable critical points. After including these induced boundary values, the construction of $W(x)$ can be alternatively extended as below.
\begin{lem}\label{lem:Wmm}
Assume there are $k$ stable local minima   of $U(x)$, interleaved by $k$ unstable local maxima periodically indexed as \eqref{id1} and \eqref{id2}.
Let $W(x)$  be defined in \eqref{FW_W} and $W_i$ be defined in \eqref{wi}. Then $W(x)$ has an alternative representation
\begin{equation}\label{wA}
W(x) = \min_i \{ W_i + h(x; x_i),\,  W(x_{i+\frac12}) + h(x; x_{i+\frac12}) \}.
\end{equation}
\end{lem}
\begin{proof}
Using definition \eqref{FW_W}, at $x_{j+\frac12}$, we have
\begin{equation}
\begin{aligned}
W(x_{j+\frac12}) + h(x; x_{j+\frac12}) &= \min_i \bbs{W_i + h(x_{j+\frac12}; x_i) + h(x; x_{j+\frac12})} \\
&\geq \min_i \bbs{W_i + h(x; x_i) } = W(x),
\end{aligned}
\end{equation}
where we used the triangle inequality $h(z; x)+h(y; z) \geq h(y; x).$ Thus we have
\begin{equation}
\min_j \bbs{W(x_{j+\frac12}) + h(x; x_{j+\frac12}) } \geq W(x). 
\end{equation}
This, together with $\min_i \{ W_i + h(x; x_i),\,  W(x_{i+\frac12}) + h(x; x_{i+\frac12}) \}\leq W(x)$ implies \eqref{wA}.
\end{proof}

\subsection{The computation of the calibrated curves and the projected Mather set $\mathcal{M}_0$}\label{sec4_kam}
After all the preparations above, we now prove $W(x)$ constructed via \eqref{FW_W} and boundary data \eqref{wi} is a weak KAM solution of negative type. 

Recall the definition of weak KAM solutions of negative type, cf.  \cite[Definiteion 4.1.11]{fathi2008weak}
\begin{defn}
We say a continuous function $u\in C(\bS^1)$ is a weak KAM solution of negative type to HJE \eqref{HJE_w} if
\begin{enumerate}[(I)]
\item $u$ is dominated by $L$ (denoted as $u\prec L$), i.e., for any absolutely continuous curve $\gamma\in AC([a,b]; \bS^1)$,  
\begin{equation}
u(\vg(b)) - u(\vg(a)) \leq \int_a^b L(\dot{\vg}, \vg) \ud t;
\end{equation}
\item for any $x$, there exists a continuous, piecewise $C^1$ curve $\vg: (-\8,0] \to \bS^1$ with $\gamma(0)=x$ such that for any $a<b \leq 0$
\begin{equation}
u(\vg(b) )- u(\vg(a))= \int_a^b L(\dot{\vg}, \vg) \ud t.
\end{equation}
\end{enumerate}
\end{defn}
\begin{rem}
Such a curve $\vg$ in condition (II) is called a calibrated curve, or a backward characteristic; see examples in Figure \ref{fig_pi} for two calibrated curves associated with $W(x)$. 
\end{rem}

{\blue
\begin{prop}\label{thm1}
Assume there are $k$ stable local minimums   of $U(x)$, interleaved by $k$ unstable local maximums indexed as \eqref{id1} and \eqref{id2}. Let $W(x)$ be defined in \eqref{FW_W}.
 Then
 \begin{enumerate}[(i)]
  \item For each $x\in \bS^1$, there exists a calibrated curve tracking back to a critical point of $U(x)$. 
  \item  The projected Mather set $\mathcal{M}_0$ is same as the projected Aubry set $\mathcal{A}=\{x_i, x_{i+\frac12};\, i=1,\cdots, k\}.$
  \end{enumerate}
\end{prop}
\begin{proof}
Recall the explicit characterization for the shape of $W(x)$. 
Given any $x\in \bS^1$, we first assume $x$ locates on a decreasing interval of $U(x)$, i.e., $x\in [x_{i-\frac12}, x_{i}]$ for some $i$. Then from conclusion (ii) in Proposition \ref{prop_W}, we know that there exists a $x^* \in [x_{i-\frac12}, x_{i}]$ such that either Case (a): $x$ belongs to a constant interval such that $W(x)\equiv W(x_{i-\frac12})$ for $x_{i-\frac12}\leq x \leq x^*;$ or Case (b): $x$ belongs to a  decreasing interval such that $W(x)= U(x) + \text{ const }$ for $x^* \leq x \leq x_i;$ see Figure \ref{fig_pi}.

For Case (a), we solve the following 	`downhill' ODE backward in time
\begin{equation}\label{ODE1}
\dot{\vg} = - U'(\vg), \quad t\leq 0; \qquad \vg(0) = x.
\end{equation}
Then we obtain a unique ODE solution $\gamma(t)$ with $\vg(-\8) = x_{i-\frac12}.$ Along this ODE solution, we verify that for any $a<b \leq 0$
\begin{equation}\label{L0L0}
 \int_a^b L(\dot{\vg}, \vg) \ud t =  \int_a^b \frac14 |\dot\vg + U'(\vg)|^2 \ud t = 0 = W(\gamma(b)) - W(\gamma(a)). 
\end{equation}

For Case (b), we solve the following `uphill' ODE backward in time
\begin{equation}\label{ODE2}
\dot{\vg} =   U'(\vg), \quad t\leq 0; \qquad \vg(0) = x.
\end{equation}
Then we obtain a unique ODE solution $\gamma(t)$ with $\vg(-\8) = x_{i}.$ Along this ODE solution, we verify that for any $a<b \leq 0$
\begin{equation}
 \int_a^b \frac14 |\dot\vg + U'(\vg)|^2 \ud t = \int_a^b \bbs{\frac14 |\dot\vg - U'(\vg)|^2 + \dot \vg U'(\vg) } \ud t = U(\vg(b)) - U(\vg(a)) = W(\gamma(b)) - W(\gamma(a)). 
\end{equation}

Therefore, for both two cases, we verified $W(x), x\in [x_{i-\frac12}, x_{i}]$ satisfies the condition (II). Similarly, if $x\in [x_{i}, x_{i+\frac12}]$ for some $i$, one can also repeat the same argument to verify $W(x), x\in [x_{i}, x_{i+\frac12}]$ satisfies condition (II).

In summary, calibrated curves have three types:  Type 1) For any differential point $x\in \bS^1$ locating on strictly increasing or decreasing part, there exists a unique backward characteristic solving \eqref{ODE2} such that $\vg(0)=x$, $\vg(-\8)$ tracks back to a unique local minimum(attractor) $x_i$ in the same basin of attraction as $x$;  Type 2) For any differential point $x\in \bS^1$ located on constant part of $W(x)$, there exists a unique backward characteristic solving \eqref{ODE1} such that $\vg(0)=x$, $\vg(-\8)$ tracks back to the local maximum $x_{i-\frac12}$ at the end of the constant segment of $W(x)$;  Type 3) For any non-differential points $x\in \bS^1$, there exist two backward characteristics either solving \eqref{ODE2} or \eqref{ODE1} and thus they track back to one of the adjacent critical points.

Consequently, based on the Aubry-Mather theory, cf.  \cite{evans2008weak, fathi2008weak},  a Mather measure concentrates on one of those extremes $\gamma(-\8)$ for the above calibrated curves and $s=0$, i.e., $\mu = \delta(x-\gamma(-\8))\delta(s)$. In detail, 
one can define $\mu_T$ for fixed $T$ as $\la f, \mu_T \ra := \frac1T\int_{-T}^0 f(\dot{\gamma}(t), \gamma(t)) \ud t$. Then taking the limit we have
$$
\int_{T \bS^1} f(s,x)   \mu(\ud s, \ud x) = \lim_{T \to +\8} \int_{T \bS^1} f(s,x)   \mu_T(\ud s, \ud x)
= \int_{T \bS^1} f(s,x)     \delta_{\dot{\vg}(-\8)} \otimes  \delta_{\vg(-\8)}.  $$
  Thus   the Mather set is given by
\begin{equation}\label{Mset}
\tilde{M}=\overline{\cup \text{ support } \mu } = \{ (x_i ,0), (x_{i+\frac12}, 0) ; \, i=1,\cdots,k\}.
\end{equation}
Hence we conclude (i) and (ii).
\end{proof}
}

Recall $W(x)$ is a viscosity solution to \eqref{HJE_w} proved in Proposition \ref{prop_W}.
Notice the weak KAM condition (I) can be directly implied from $W(x)\in \text{Lip}(\bS^1)$ being a   viscosity subsolution satisfying $H( W'(x), x) \leq 0$ a.e. $x\in \bS^1$; see Proposition \ref{prop_W}. Indeed, for any absolutely continuous curve $\gamma(\cdot)$ with $\vg\in AC([a,b];\bS^1)$, we have
\begin{equation}
\begin{aligned}
W(\vg(b)) - W(\vg(a)) =& \int_a^b W'(\vg(s)) \cdot \dot{\vg}(s) \ud s\\
 \leq &  \int_a^b \bbs{L(\dot{\vg}(s), \vg(s)) + H(W'(\vg(s)),\vg(s))  } \ud s \leq \int_a^b  L(\dot{\vg}(s), \vg(s)) \ud s.
 \end{aligned}
\end{equation}
The    maximality of $W(x)$ is the same as Proposition \ref{prop_maximalL}, where the boundary data $W_j$ is given only on a  subset of the projected Aubry set. Thus we refer to the proof of Proposition \ref{prop_maximalL}. This, together with Proposition \ref{thm1}, yields
{\blue \begin{cor}\label{cor_kam}
Assume there are $k$ stable local minimums   of $U(x)$, interleaved by $k$ unstable local maximums indexed as \eqref{id1} and \eqref{id2}. Let $W(x)$ be defined in \eqref{FW_W}.
 Then
$W(x)$   is a weak KAM solution of negative type to HJE \eqref{HJE_w}. 
  And $W(x)$  
  is the  maximal Lipschitz continuous viscosity solution to \eqref{HJE_w} that satisfying boundary data $W(x_j)=W_j$ at $x_j$, $j=1, \cdots,k$.
\end{cor}
}

{\blue
We point out that we only need the boundary data on the local minima of $U$, i.e., a subset of the projected Aubry set, to select a meaningful weak KAM solution which serves as the rate function in the large deviation principle for invariant measures. On the other hand, given any boundary data in a subset of the projected Aubry set, one can construct a maximal Lipschitz continuous viscosity solution; see Proposition \ref{prop_maximalL}.
}

\begin{rem}
From \eqref{L0L0}, it is easy to see that for any curve $\gamma:\bR \to \bS^1$ which solves the ODE $\dot{\vg} = \pm U'(\vg), t\in \bR$, is a least action curve. This gives the projected Ma\~n\'e set 
$$\overline{\cup \{\vg(t); \vg \text{ solves } \dot{\vg} = \pm U'(\vg), t\in \bR\}} = \bS^1.$$
The Ma\~n\'e set itself is the collection of the Lagrangian graph $(\vg, \pm U(\vg))$ of those least action curves. Furthermore, in this simple example, it is easy to see the Mather set \eqref{Mset} is a compact Lipschitz graph, and is invariant under the Euler-Lagrange flow. This is the essence of the celebrated  Mather graph theorem that characterizes the graph property of the Mather set.
\end{rem}

\subsubsection{Invariant solutions of the Lax-Oleinik semigroup}

In this section, using the equivalent characterization of   invariant solutions of the Lax-Oleinik semigroup \cite[Proposition 4.6.7]{fathi2008weak},  we give a direct corollary that $W(x)$ defined in \eqref{FW_W} is   an invariant solution for the Lax-Oleinik semigroup $S_t$ associated with the dynamic HJE $\pt_t u + H(\pt_x u(x), x)=0$, i.e., for $t\geq 0$,
\begin{equation}
(S_t u_T)(y):=   \inf_x\bbs{u_T(x) + \inf_{\vg; \gamma(0)=x, \, \vg(t)=y} \int_0^t L(\dot{\vg}, \vg) \ud \tau}, \quad \forall u_T(x).
\end{equation}
\begin{cor}\label{cor_st}
Any weak KAM solution $w(x)$ to HJE \eqref{HJE_w} is an invariant solution of the Lax-Oleinik semigroup $S_t$ and satisfies the representation
\begin{equation}\label{res}
w(y) = \inf_{x\in\bS^1} \bbs{w(x) + v(y;x)} = \inf_{x_i \in \mathcal{A}} \bbs{w(x_i) + h(y; x_i)}.
\end{equation}
Particularly,  $W(x)$   defined in \eqref{FW_W} is an invariant solution of  the Lax-Oleinik semigroup $S_t$.
\end{cor}
\begin{proof}
First, from \cite[Proposition 4.6.7]{fathi2008weak}, any weak KAM solution $w(x)$ is an invariant solution of the Lax-Oleinik semigroup $S_t$. Thus
$$w(y) = \inf_x \bbs{w(x) + \inf_{\vg; \gamma(0)=x, \, \vg(t)=y} \int_0^t L(\dot{\vg}, \vg) \ud \tau}.$$
Taking infimum w.r.t. $t$ and exchanging   $\inf_x$ and $\inf_t$, we obtain $w(y) = \inf_x \bbs{w(x) + v(y;x)}$.

Second, take the boundary values $w(x_i)$ on the projected Aubry set $\mathcal{A}$. Since the projected Aubry set is a uniqueness set for the weak KAM solution \cite[Theorem 4.12.6]{fathi2008weak}, these boundary values can uniquely define a weak KAM solution. Meanwhile, from Corollary \ref{cor_kam} $w(x)=\inf_{x_i \in \mathcal{A}} \bbs{w(x_i) + h(y; x_i)}$ is a weak KAM solution and thus the representation \eqref{res} holds uniquely.
\end{proof}
\begin{rem}
We remark for a compact domain, the existence of   invariant solutions of $S_t$ and the convergence from dynamic solution to   an invariant solution were proved in \cite{fathi2008weak, namah1999remarks}. However, the  invariant solutions are not unique, as well as the weak KAM solutions; see Examples in Figure \ref{fig_w000} in the next section and \cite{fathi2009hausdorff}. 
\end{rem}

\subsection{Generating a set of consistent boundary data and constructing a   maximal Lipschitz continuous viscosity solution}\label{sec4_gene}
In this subsectiton, given any non-consistent boundary data on a subset of the projected Aubry set, we can  first use it to generate a set of consistent data satisfying the discrete weak KAM problem. Then based on these consistent data, we prove the variational formula $W(x)$ is the maximal Lipschitz continuous viscosity solution to the HJE satisfying the generated boundary data.
\subsubsection{Non-consistent boundary data induce a set of consistent data satisfying  \eqref{disW}}
Now given any boundary data 
\begin{equation}
\{W_{\ell} \} \quad \text{ at } D:=\{x'_\ell; \,\, \ell = 1,2, \cdots, m\} \subset \mathcal{A}, 
\end{equation}
 which may not satisfy the discrete weak KAM problem \eqref{disW}. The following procedure  can be used to obtain a set of consistent boundary data satisfying \eqref{disW}.
 For any $j=1, \cdots, m$, define
 \begin{equation}\label{it}
  W(x) =\min_{\ell=1, \cdots, m} \{W_\ell + h(x;x'_{\ell})\}.
 \end{equation}
Then $\tilde{W}_\ell:= W(x'_\ell)$ is a set of consistent data satisfying discrete weak KAM problem
\begin{equation}\label{disWn}
\tilde{W}_j = \min_{\ell=1, \cdots, m} \{\tilde{W}_\ell + h(x'_{j};x'_{\ell})\}, \quad j = 1, \cdots, m.
\end{equation}

\begin{prop}\label{prop_gene}
Given any boundary data $\{W_{\ell} \}$ on a subset $\{x'_\ell; \,\, \ell = 1,2, \cdots, m\}$ of $\mathcal{A}$, then 
\begin{equation}
\tilde{W}_j: = \min_{\ell=1, \cdots, m} \{W_\ell + h(x'_j;x'_{\ell})\}, \quad j=1, \cdots, m
\end{equation}
  satisfy the discrete weak KAM problem \eqref{disWn}.
\end{prop}
\begin{proof}
Given boundary data $\{W_{\ell} \}$, $\ell=1, \cdots, m$, let $W(x)$ be defined in \eqref{it}.

On the one hand, $\tilde{W}_j:= W(x'_j)\leq W_j$ for any $j=1, \cdots, m.$ 
Thus 
\begin{equation}
\tilde{W}(x) = \min_{\ell=1, \cdots, m} \{\tilde{W}_\ell + h(x;x'_{\ell})\} \leq W(x).
\end{equation}

On the other hand, 
 from Proposition \ref{prop_W}, $W(x)$ is a Lipschitz continuous viscosity solution, while from Proposition \ref{prop_h}, $\tilde{W}_\ell + h(x;x'_{\ell})$ is the maximal Lipschitz continuous viscosity solution. Thus from $W(x'_\ell)=\tilde{W}_\ell$,  we have
\begin{equation}
W(x) \leq \min_{\ell=1, \cdots, m} \{ \tilde{W}_\ell + h(x;x'_{\ell}) \} = \tilde{W}(x).
\end{equation}

Therefore, we conclude $\tilde{W}(x) = W(x)$. Particularly, $\tilde{W}(x'_\ell) = W(x'_\ell) =\tilde{W}_\ell$ for $\ell=1, \cdots, m$ and thus $\tilde{W}_\ell$ is a consistent boundary value satisfying \eqref{disWn}.
\end{proof}
We point out in the above proposition, the subset $\{x'_\ell; \,\, \ell = 1,2, \cdots, m\}$ is not necessarily  a uniqueness set. Indeed, in the next subsection, we will prove that as long as the boundary values $W_{\ell}$ satisfies the discrete weak KAM problem \eqref{disWn}, then we can use those data to obtain a maximal Lipschitz continuous viscosity solution. Particularly, the weak KAM solution in Corollary \ref{cor_kam} is the maximal Lipschitz continuous viscosity solution satisfying boundary data \eqref{wi}.

\subsubsection{Maximal Lipschitz continuous viscosity solution based on consistent data}
In the next proposition, given any boundary values $W_{\ell}$, for   $ D:=\{x'_{\ell}; \,\, {\ell=1,\cdots,m}\} \subset \mathcal{A},$ if $\{W_\ell;\, \ell=1, \cdots, m\}$ satisfies the discrete weak KAM problem \eqref{disWn},
 we prove   $W(x)$ with the representation 
\begin{equation}\label{Wc}
W(x) = \min_{\ell =1, \cdots, m} \{W_\ell + h(x; x'_{\ell})\}
\end{equation}
 is indeed the maximal Lipschitz continuous viscosity solution to the HJE satisfying  given boundary values $W_\ell$, i.e.,
 \begin{equation}\label{HJE_ma}
H(W'(x),x) = W'(W'-U')=0, \quad x\in \bS^1; \quad W(x'_\ell) = W_\ell  \text{ for } x'_\ell\in D.
\end{equation}
Consequently, $W(x)$ constructed in \eqref{FW_W} is not only one of the weak KAM solution satisfying given boundary condition on all local minimums but also   the maximal Lipschitz continuous viscosity solution to $H(W'(x),x) = W'(W'-U')=0$ with those given boundary conditions.
\begin{prop}\label{prop_maximalL}
Given any boundary values $W_\ell$ on  $D=\{x'_{\ell}; \,\, {\ell=1,\cdots,m}\}\subset \mathcal{A},$ the  solution $W(x)$ constructed via 
\eqref{Wc} is the  maximal Lipschitz continuous viscosity solution to \eqref{HJE_ma}.
\end{prop}
\begin{proof}
From Step 4 in the proof of Proposition \ref{prop_h}, we know for any $\ell=1, \cdots, m$, the lifted Peierls barrier $W_\ell+ h(x;x'_\ell)$   is the maximal Lipschitz continuous viscosity  solution to  $W'(W'-U')=0$ satisfying the boundary value $W(x'_\ell) = W_\ell$. Given any Lipschitz continuous viscosity solution $\tilde{W}(x)$ to \eqref{HJE_ma}, since $\tilde{W}(x)$ is a Lipschitz continuous viscosity subsolution, we know
\begin{equation}
\tilde{W}(x) \leq   W_\ell + h(x; x'_\ell) 
\end{equation}
for any $x'_\ell\in D.$
Notice $\tilde{W}(x)$ satisfies all the boundary values $\tilde{W}(x'_\ell)= W_\ell$ for $x'_\ell\in D$, hence taking minimum for all $x'_\ell\in D$, we obtain
\begin{equation}
\tilde{W}(x) \leq  \min_{\ell=1, \cdots, m} \{W_\ell + h(x; x'_\ell)\} = W(x).
\end{equation}
Thus, together with Proposition \ref{prop_W},  $W(x)$ is the maximal Lipschitz continuous viscosity solution satisfying all the boundary values $W(x'_\ell)=W_\ell$ on $D$.
\end{proof}

\section{Selection principles for weak KAM solutions and probability interpretation}\label{sec4_nonunique}
{\blue In this section, we discuss the selection principle given by the global energy landscape in the large deviation principle for invariant measures and also compare it with another selection principle via the vanishing discount limit. It is well known that the dynamic HJE has a unique viscosity solution and the long time limits exist but are not unique. That is to say, a selection principle is needed for viscosity solutions to stationary HJE, even the Hamiltonian is strictly convex.  Among all the viscosity solutions, the weak KAM solutions are the maximal Lipschitz continuous viscosity solution satisfying specific boundary conditions on a uniqueness set. Hence weak KAM solution serves as a natural candidate for the selection principle and the key point to select a meaningful weak KAM solution is  the determination of boundary data on a uniqueness set of the weak KAM solutions. The variational formula for those boundary data $W(x_i)$ obtained in \eqref{wi} gives a unique determination that captures the asymptotic behaviors of the original stochastic process at each local attractors.  Equivalently, we summarize it as a selection principle for those weak KAM solutions by exchanging the double limits for $t$ and $\eps$. That is we  first take the long time limit $\lim_{t\to +\8} \rho_\eps(x,t)$ which is unique due to ergodicity  and then take the zero noise limit $\eps \to 0$ due to the large deviation principle for invariant measures. In general, the vanishing viscosity limit is an approximation method for stationary HJE, but the limit is only in the subsequence sense and is not unique. Our selection principle provides a special viscosity approximation to the stationary HJE whose vanishing viscosity limit is unique. The probability
interpretations via the Boltzmann analysis of the global energy landscape from the weak KAM perspective are also discussed. }

\subsection{Examples for non-uniqueness}\label{sec5.1}
In the proof of  Corollary \ref{cor_kam}, we did not use the explicit values of $W(x_i)$ at the local minima $x_i$. Indeed, given any boundary values $W(x_i)$ for any subset of (not necessarily all) those local minima $x_i$, as long as those boundary values   are consistent with the associated discrete weak KAM problem \eqref{disW}, then $W(x)$ determined by those given boundary values though \eqref{FW_W} is a weak KAM solution.
 
 Furthermore, we use a classical example, which appears in the first edition of the book  \cite[Section 6.4]{FW} in 1979,  to  illustrate the boundary values consistent with the discrete weak KAM problem \eqref{disW} is not unique. 
 
 Choose a skew periodic potential $U(x)$ such that $U$ has $3$ local minima $x_1, x_2, x_3$ with values $1,0,2$ and has $4$ local maxima $x_{\frac12}, x_{1+\frac12}, x_{2+\frac12}, x_{3+\frac12}$ with values $7,5,10,11.$ In Figure \ref{fig_w000} left/right, the original skew periodic potential $U(x)$ is the same as the one in Figure \ref{fig_p_m}, and two plots for $W^*(x)$ with different boundary values based on \eqref{W_ex1} and \eqref{W_ex2} are shown for comparison.
 \begin{figure} 
\includegraphics[scale=0.35]{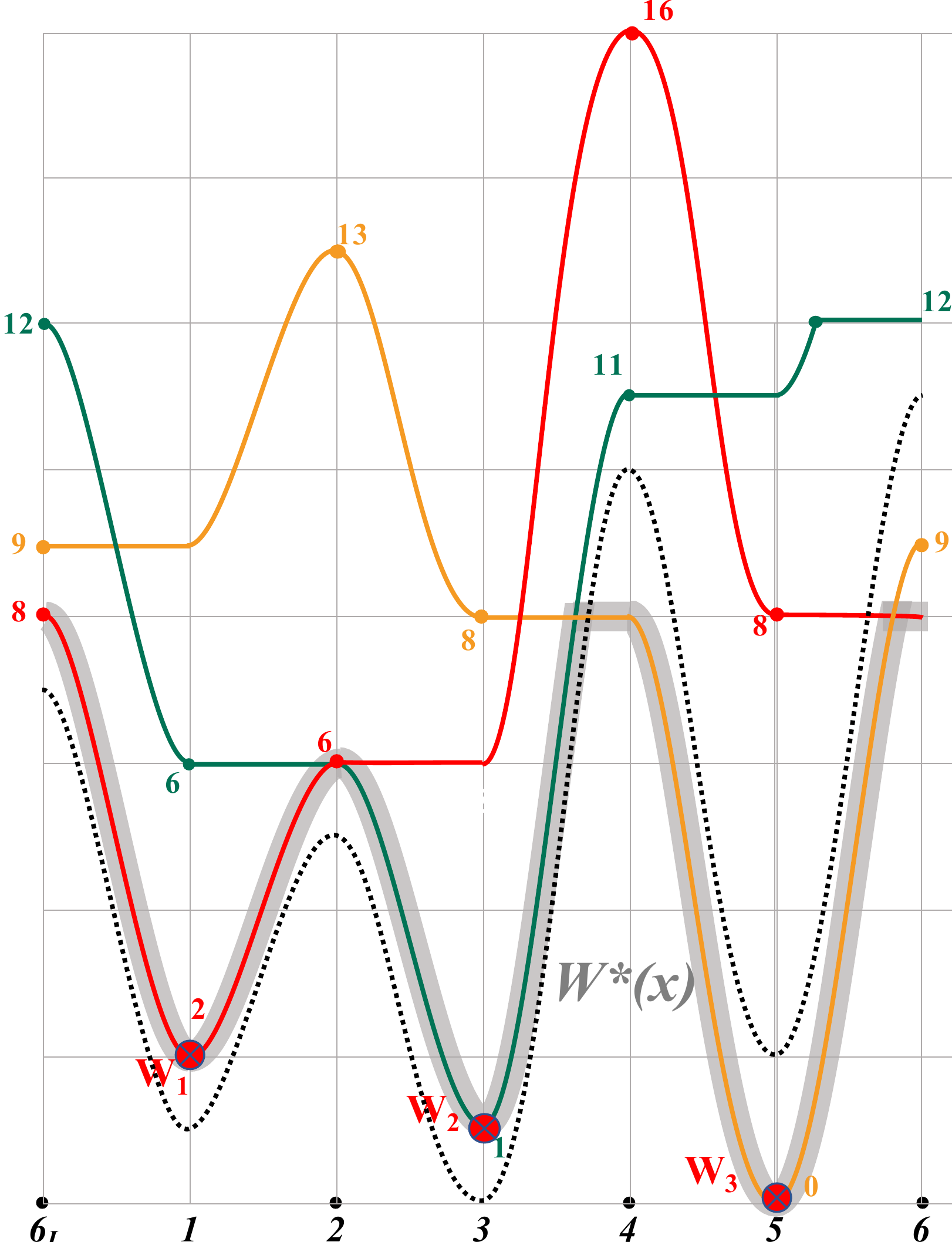} \,\,\,
\includegraphics[scale=0.35]{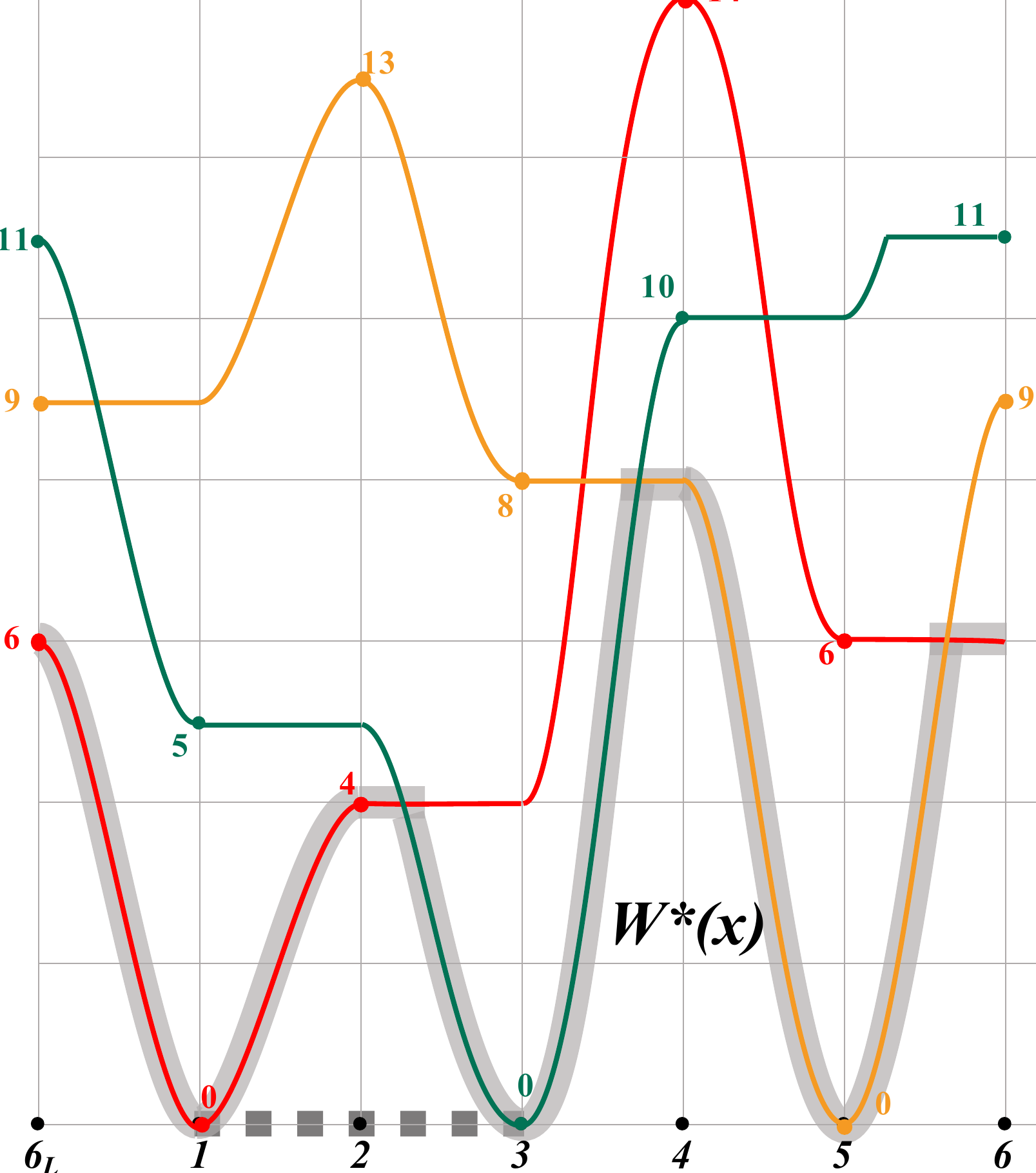}
\caption{Two examples for $W^*(x)$ computed from the variational formula \eqref{FW_W} with different consistent boundary data. (Left) The boundary values $W_1, W_2, W_3$ satisfies Freidlin-Wentzell's variational formula for boundary data \eqref{wi}. The dashed line is the original potential $U(x)$ and one can see the global adjustment for $W_1, W_2, W_3$. The explicit formula for $W^*(x)$ in \eqref{W_ex1} is marked with grey strip. (Right) $W^*(x)$ constructed with $W(x_1)= W(x_2)= W(x_3)=0$ as boundary data, which does not satisfies \eqref{wi}. If one add an additional boundary condition  $W(x_{1+\frac12})=0$, then one obtain a different solution which vanishes in the interval $[x_1,x_2]$; shown as dashed line. Both figures use the same Peierls barriers: $h(y;x_1)$ shown in red line, $h(y;x_2)$ shown in green
 line, and $h(y;x_3)$ shown in orange line.  At each connection point, only two lifted Peierls barriers turn on to finish the local trimming procedures, as described in Proposition \ref{prop_local}.   }\label{fig_w000} 
 \end{figure}  

 One has a set of boundary values computed from \eqref{wi}, $W(x_1)=13, W(x_2)=12, W(x_3)=11.$
 It is easy to verify this set of boundary values satisfy the discrete weak KAM problem \eqref{disW}. Then from \eqref{FW_W} and \eqref{rate}, $W(x)$ is given by
 \begin{equation}\label{W_ex1}
 W^*(x) = W(x) - \min_i W(x_i) = \left\{ 
 \begin{array}{cc}
 1+U(x), & x\in[x_{\frac12}, x_2];\\
  \min\{1+U(x), 8\}, & x\in[x_2, x_{2+\frac12}];\\
    U(x)-2, & x\in[x_{2+\frac12}, x_3];\\
    \min\{U(x)-2, 8\}, &x\in[x_3, x_{3+\frac12}],
 \end{array}
 \right.
 \end{equation}
 which  satisfies Proposition \ref{prop_W}.
 
 Another set of boundary values can be chosen as $W(x_1)= W(x_2)= W(x_3)=0.$
 It is easy to verify this set of boundary values also satisfy the discrete weak KAM problem \eqref{disW}. Then from \eqref{FW_W}, $W(x)$ is given by
 \begin{equation}\label{W_ex2}
 W^*(x)=W(x) = \left\{ 
 \begin{array}{cc}
  U(x)-1, & x\in[x_{\frac12}, x_{1+\frac12}];\\
  \min\{ U(x), 4\}, & x\in[x_{1+\frac12}, x_{2}];\\
    \min\{ U(x), 8\}, & x\in[x_{2}, x_{2+\frac12}];\\
     U(x)-2, &x\in[x_{2+\frac12}, x_{3}];\\
      \min\{ U(x)-2, 6\}, & x\in[x_{3}, x_{3+\frac12}].\\
 \end{array}
 \right.
 \end{equation}
 which also  satisfies Proposition \ref{prop_W}. From Corollary \ref{cor_kam}, both sets of boundary values give a weak KAM solution to \eqref{HJE_w}, so   weak KAM solutions are not unique.

\subsection{Exchange limits $\eps \to 0$, $t\to +\8$ in two large deviation principles}
Below we discuss a special case for which the long time behavior limit $t\to +\8$ and the zero noise limit $\eps \to 0$ for the diffusion process \eqref{LV} can be exchanged. Notice that in general, it is not exchangeable.

{\blue Recall the Fokker-Planck equation on $\bS^1$ \eqref{FPt} and WKB reformulation $\rho_\eps(x,t) = e^{-\frac{\psi_\eps(x,t)}{\eps}} $. The viscous HJE associated with $\psi_\eps$ is  
\begin{equation}\label{HJEpsi}
\pt_t \psi_\eps(x,t) + H(\pt_x \psi_\eps(x,t), x) = \eps \bbs{\pt_{xx}\psi_\eps(x,t) - U''(x)}, \quad x\in \bS^1,\, t>0.
\end{equation}
In general, the two limits for $\psi_\eps(x,t)$ as $\eps \to 0$ and $t\to +\8$ can not be exchanged. But with special initial data, we have the following result.}
\begin{prop}\label{prop_exchange}
Let   $W^*(x)$ be the rate function given by \eqref{rate} and assume $U(x)$ have finite local extremes. 
Assume $\rho_\eps(x,t)$ is a solution to the Fokker-Planck equation \eqref{FPt} and the initial distribution is given by $\rho_\eps(x,0) = C_\eps e^{-\frac{W^*(x)}{\eps}}$. Then
\begin{equation}\label{limit12}
\lim_{\eps \to 0} \bbs{-\eps \log \lim_{t\to +\8} \rho_\eps(x,t) }  \text{ ``}=\text{'' }  W^*(x) = \lim_{t\to +\8} \lim_{\eps \to 0} - \eps \log \rho_\eps(x,t).
\end{equation}
Here  $	\text{ ``}=\text{'' }$ is in the large deviation principle sense \eqref{LD2}.
\end{prop}
\begin{proof}
On the one hand, for fixed $\eps$, the ergodicity $\lim_{t\to +\8}\rho_\eps(x,t)=\pi_\eps(x)$ is a standard result for over-damped Langevin dynamics on $\bS^1$. Thus from the large deviation principle \eqref{LD2} \cite[Chapter 6, Theorem 4.3]{FW}, we have
\begin{equation} \label{ld_vis}
 \lim_{\eps \to 0} \bbs{-\eps \log \lim_{t\to +\8} \rho_\eps(x,t) }    \text{ ``}=\text{'' }  W^*(x),
 \end{equation}
 where $	\text{ ``}=\text{'' }$ is in the large deviation principle sense \eqref{LD2}.
 
 On the other hand, $\psi_\eps(x,t) = - \eps \log \rho_\eps(x,t)$ is the solution to the HJE \eqref{HJEpsi} with initial data $W^*(x)- \eps \log C_\eps$. Since $\min W^*(x)=0$, so by the Laplace  principle
  $\int e^{-\frac{W^*(x)}{\eps}} \sim O(1)$ as $\eps \to 0$. Thus $C_\eps \sim O(1)$ and as $\eps \to 0$, $W^*(x)- \eps \log C_\eps \to W^*(x)$. From \cite{crandall1983viscosity, crandall1984two} vanishing viscosity method,  we know the convergence from the solution $\psi_\eps(x,t)$ of \eqref{HJEpsi} to the viscosity solution $\psi(x,t)$ of the limiting first order HJE 
  $$\pt_t \psi(x,t) + H(\pt_x \psi (x,t), x)=0, \quad \psi(x,0)=W^*(x).$$
  Then by the Lax-Oleinik semigroup representation 
  \begin{equation}
\psi(x,t)=(S_t W^*)(x)=  W^*(x),
\end{equation}
where we used  $W^*(x)$ is an invariant solution due to Corollary \ref{cor_st}.
Thus we know $$W^*(x) = \lim_{t\to +\8} W^*(x) = \lim_{t\to +\8} \lim_{\eps \to 0} - \eps \log \rho_\eps(x,t).$$
\end{proof} 
We remark the exchanging of two limits on the left and right hand sides is in general incorrect. Indeed, the limits in the left hand side of \eqref{limit12} is unique. This is because the invariant measure for $t\to +\8$ exists and is unique. Then the rate function of the large deviation principle for invariant measures is unique. However, the right hand side first finds the rate function for the large deviation principle as $\eps \to 0$ at finite time, which   solves a dynamic HJE. Then the long time limit $t\to +\8$ for the dynamic solution exists \cite{BSBS} but in general is not unique. Therefore, a selection principle is needed, and particularly the limits on the left hand side provides a meaningful selection principle for stationary HJE via the large deviation principle for invariant measures. Below, we discuss two selection principles: large deviation principle v.s. vanishing discount limit.
 
\subsection{Two selection principles: large deviation principle v.s. vanishing discount limit} \label{sec5.3}
 A selection principle is to give a meaningful principle to determine boundary values on the projected Aubry set $\mathcal{A}$. The global energy landscape $W(x)$ in \eqref{FW_W}, particularly the globally adjusted boundary values  on the local minima \eqref{wi},   is constructed so that $W(x)$ is the rate function for the large deviation principle of the invariant measure for the diffusion process on $\bS^1$ \cite[Chapter 6, Theorem 4.3]{FW}. That is to say, the large deviation rate function $W(x)$ for the diffusion process serves as a selection principle for weak KAM solutions. 
This selection principle could also apply to other Hamiltonian dynamics with an underlying stochastic process and a large deviation principle. We formally describe this framework below for a chemical reaction process with a random-time changed Poison representation, cf.   \cite{Kurtz15, GL22}. For any fixed time $0<t<+\8$, the large deviation principle for the  chemical reaction process in the thermodynamic limit was proved in \cite{GL22t} by using the convergence from the Varadhan's discrete nonlinear semigroup $w_\eps(x,t)$ to the viscosity solution $w(x,t)$ of the dynamic HJE, which has a Lax-Oleinik semigroup representation. If this Lax-Oleinik semigroup has an invariant solution, denoted as $w(x)$. Then this invariant solution is a weak KAM solution and has the representation $w(x)=\inf_{y\in \bR^d}(w(y)+v(x;y))$ via the Ma\~n\'e potential $v(x;y)$; see \cite[Proposition 6.11, Theorem 7.5]{Tran21} \cite[Proposition 4.6.7]{fathi2008weak} for proofs for a periodic domain. Notice these invariant solutions are in general not unique. 
However, since the Lagrangian $L(s,x)$ in the least action problem $v(x;y)$ is always nonnegative and it is proved in \cite{GL22t} that the zero-cost flow (a.k.a. the dynamics following the law of large numbers) is given by $\dot{x}=\pt_p H(0,x)$. Thus the projected Aubry set $\mathcal{A}$, which is assumed to contain only finite many points,    can be characterized by using the roots of $\pt_p H(0,x)=0$. Then the weak KAM representation can be reduced to $w(x) = \inf_{y\in \mathcal{A}}(w(y)+h(x;y))$ \cite[Theorem 7.40]{Tran21}.  Assume furthermore $w(x), x\in \mathcal{A}$ is chosen such that $w(x)$ is the rate function for the invariant measure of the chemical reaction process, then this gives a selection principle to those weak KAM solutions.
 
{\blue We comment on the stationary HJE \eqref{HJEpsi} for $\psi_\eps$ with viscosity terms can be one non-trivial viscosity approximation for the stationary HJE. In general, we know the   nonuniqueness for the vanishing viscosity limit of stationary HJE. How to construct a vanishing viscosity approximation to stationary HJE which has a unique limit for all vanishing $\eps$ is still an open question\cite{Tran21}.
In our example, thanks to the inhomogeneous term $\eps U''(x)$ in 
\begin{equation}
H(\pt_x \psi_\eps(x,t), x) = \eps \bbs{\pt_{xx}\psi_\eps(x,t) - U''(x)},
\end{equation}
one has a non-trivial solution but also has a uniform limit as $\eps \to 0$; see \eqref{ld_vis}. This serves as a meaning vanishing viscosity approximation but in general, we do not have an answer.
 } 
 
 In another direction, a selection principle is given by choosing the boundary values on the projected Aubry set $\mathcal{A}$  so that the weak KAM solution $W(x)$ is the unique viscosity solution which is the vanishing discount limit of the solution $\psi_\lambda$ to $\lambda \psi_\lambda + H(\nabla \psi_\lambda(x),x)=0$ as $\lambda \to 0.$ This direction has been widely studied in both compact or non-compact domains \cite{Contreras2001, gomes2008generalized,  Fathi16, ishii2020vanishing}. {\blue We refer to \cite{cgmt15, mitake2017selection} which include a degenerate diffusion term in the vanishing discount limit problem, and to \cite{ITM2017} for a duality framework in the vanishing discount problem for fully nonlinear, degenerate elliptic Hamiltonian. The vanishing discount limit method is different from the vanishing viscosity limit we constructed.    Particularly, for our one dimensional example on $\bS^1$, the vanishing discount  limit of the  discounted HJE with the same Hamiltonian
 \begin{equation}
 \lambda \psi_\lambda +  \psi_\lambda' (\psi_\lambda'  - U')=0
 \end{equation} 
 is trivial. For $\lambda>0$, there is a unique viscosity solution $\psi_\lambda\equiv 0$ due to the comparison principle. Thus  its vanishing discount limit $\lim_{\lambda \to 0^+} \psi_{\lambda} \equiv 0$ is  the selected weak KAM solution to \eqref{HJE_w} via the vanishing discount  limit.} 
 
 Based on the discussions above, we can see at least, for the diffusion process on $\bS^1$, the 
  vanishing discount limit  and the rate function in large deviation principle are two different selection principles which result to different weak KAM solutions. This is analogous to the idea that in general the two limits $t\to +\8$ and $\eps \to 0$ for \eqref{HJEpsi} are  non-exchangeable.

  \subsection{ Boltzmann analysis  for the weak KAM solution $W^*(x)$ selected via large deviation principle }\label{sec4_prob}
In this section, based on the weak KAM solution $W^*(x)$ defined in \eqref{FW_W} with boundary data $W(x_i)=W_i$ constructed in \eqref{wi},   we elaborate  some probability interpretations that can be explained or computed via the weak KAM solution properties.
 
The classical Boltzmann analysis in statistical mechanics shows that in an equilibrium system, the probability for a particle being at a certain state $x$ is a function of the state's energy $E(x)$ and the temperature $T$
\begin{equation}
\pi(x) \propto e^{-\frac{E(x)}{k_B T}}.
\end{equation} 
Then the ratio of the probability between any two states is
\begin{equation}
\frac{\pi(x_1)}{\pi(x_2)} = e^{\frac{E(x_2)-E(x_1)}{k_B T}}.
\end{equation} 
However, for a non-equilibrium system, for instance the irreversible diffusion example on $\bS^1$ \eqref{FP}, this ratio can not be computed directly from the original  potential energy $U(x)$.

Indeed, the weak KAM solution $W(x)$ provides the answer, which not only serves as the good rate function of the large deviation principle of invariant measure $\pi_\eps(x)$ but also allows one to find a calibrated curve for any $x$ tracking back to a critical point in the projected Aubry set $\mathcal{A}$.
This calibrated curve allows one to compute the ratio of the probabilities between the starting point $\vg(0)=x^*$ and its reference point $\vg(-\8)$.
\begin{equation}
\frac{\pi_\eps(x^*)}{\pi_\eps (x(-\8))} \approx e^{\frac{W(x(-\8)) - W(x^*)}{\eps}}.
\end{equation}
The value of this ratio, depending on the explicit calibrated curve starting from $x^*$, is either $1$ or $e^{\frac{U(x_i) - U(x^*)}{\eps}}$. These ratios of probabilities w.r.t. different reference points due to different calibrated curves are shown in Figure \ref{fig_pi}.
\begin{figure}
\includegraphics[scale=0.43]{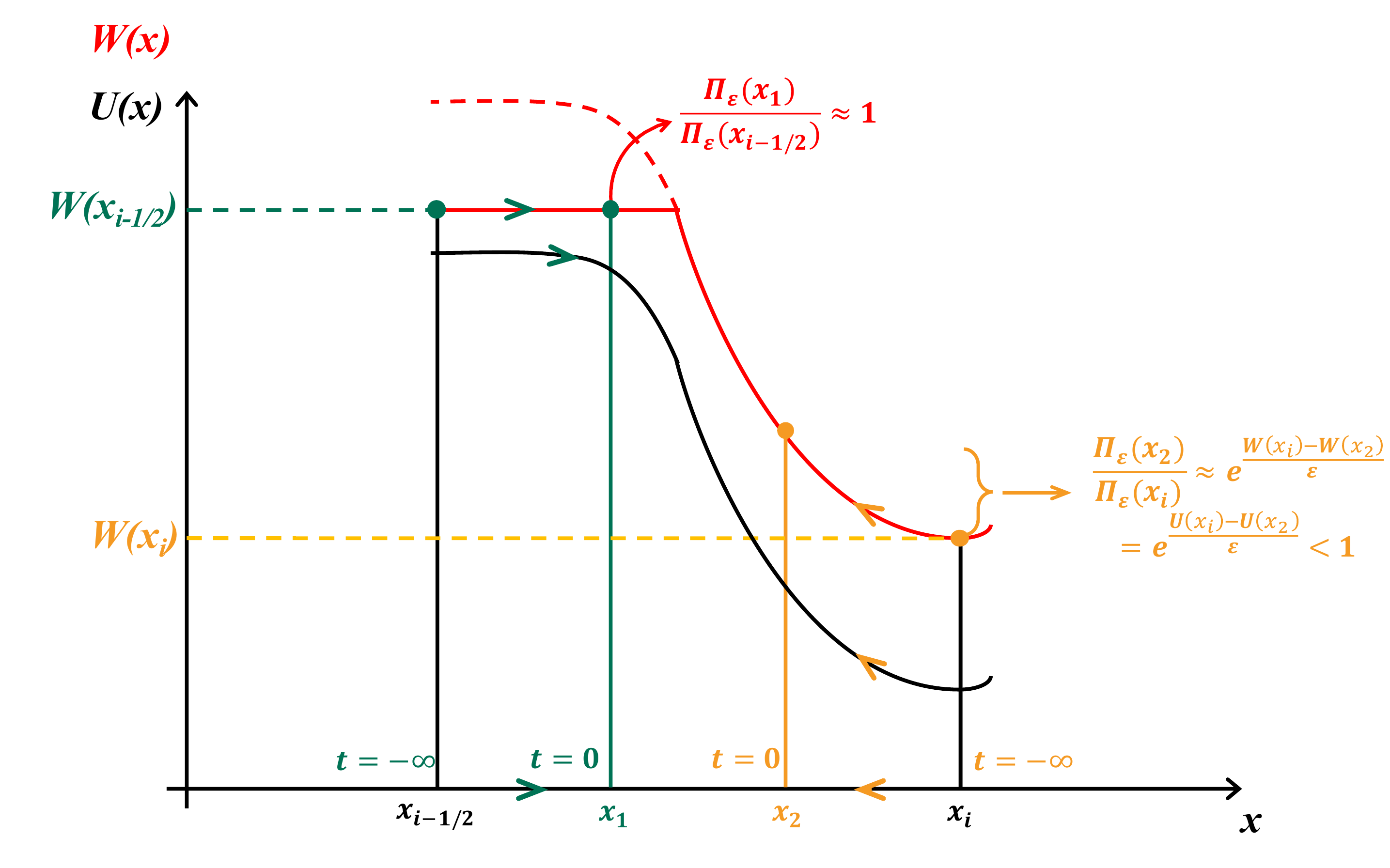}
\caption{An illustration of the weak KAM solution $W(x)$   and two calibrated curves starting from $x_1$ and $x_2$ backward in time. The solid black line is the original potential $U(x)$ while the solid red line is the weak KAM solution $W(x)$, which serves as the global energy landscape in the zero noise limit. The calibrated curve (green arrow) starting from $x_1$ solves `downhill' ODE $\dot{\vg}=-U'(\vg)$ \eqref{ODE1} backward in time and tracks back to $x_{i-\frac12}$ with the same level in the global energy landscape $W(x)$. The ratio of the probabilities at $x_1$ and its reference point $x_{i-\frac12}$ equals one, which indicates these two states appear  with the same probability in the zero noise limit. The    calibrated curve (orange arrow) starting from $x_2$ solves `uphill' ODE $\dot{\vg}=U'(\vg)$ \eqref{ODE2} and tracks back to its reference point $x_{i}$. The probability ratio $\frac{\pi_\eps(x_2)}{\pi_\eps(x_i)}$, in the zero noise limit, is smaller than $1$ and indicates state $x_2$ appear less likely than $x_i$.}\label{fig_pi} 
\end{figure}

 \section*{Acknowledgements}
The authors would like to thank Jin Feng and Hung Tran for valuable  suggestions. Yuan Gao was supported by NSF under Award DMS-2204288. J.-G. Liu was supported   by NSF under award DMS-2106988.

 \appendix
 
 \section{Remarks on Ma\~n\'e potential is not a viscosity solution on $\bS^1$}
 
Regarding the conclusions (iii) and (iv) in the Proposition \ref{prop_h}, we emphasize that the non-differential point cannot be resulted from a $C^1$ function cut off from below by a constant, otherwise it is not a viscosity solution to HJE. 
Indeed, from the proof of (iv), if a $C^1$ function is cut off from below by a constant, then at the non-differential point,   a constant is connected to an increasing function, where $D^+ h(x^*)=\emptyset, \, D^- h(x^*)= \{q; 0\leq q \leq h'(x^*_+)=U'(x^*)\}$. Then it's easy to verify $q( q-U'(x) )\leq 0$ does not satisfies the viscosity supersolution test; see Figure \ref{fig_p_m} for the comparison of the shape of the Peierls barrier $h(y;x_i)$ and a general 
Ma\~n\'e potential $v(y; x^{**}).$
\begin{figure}
\includegraphics[scale=0.27]{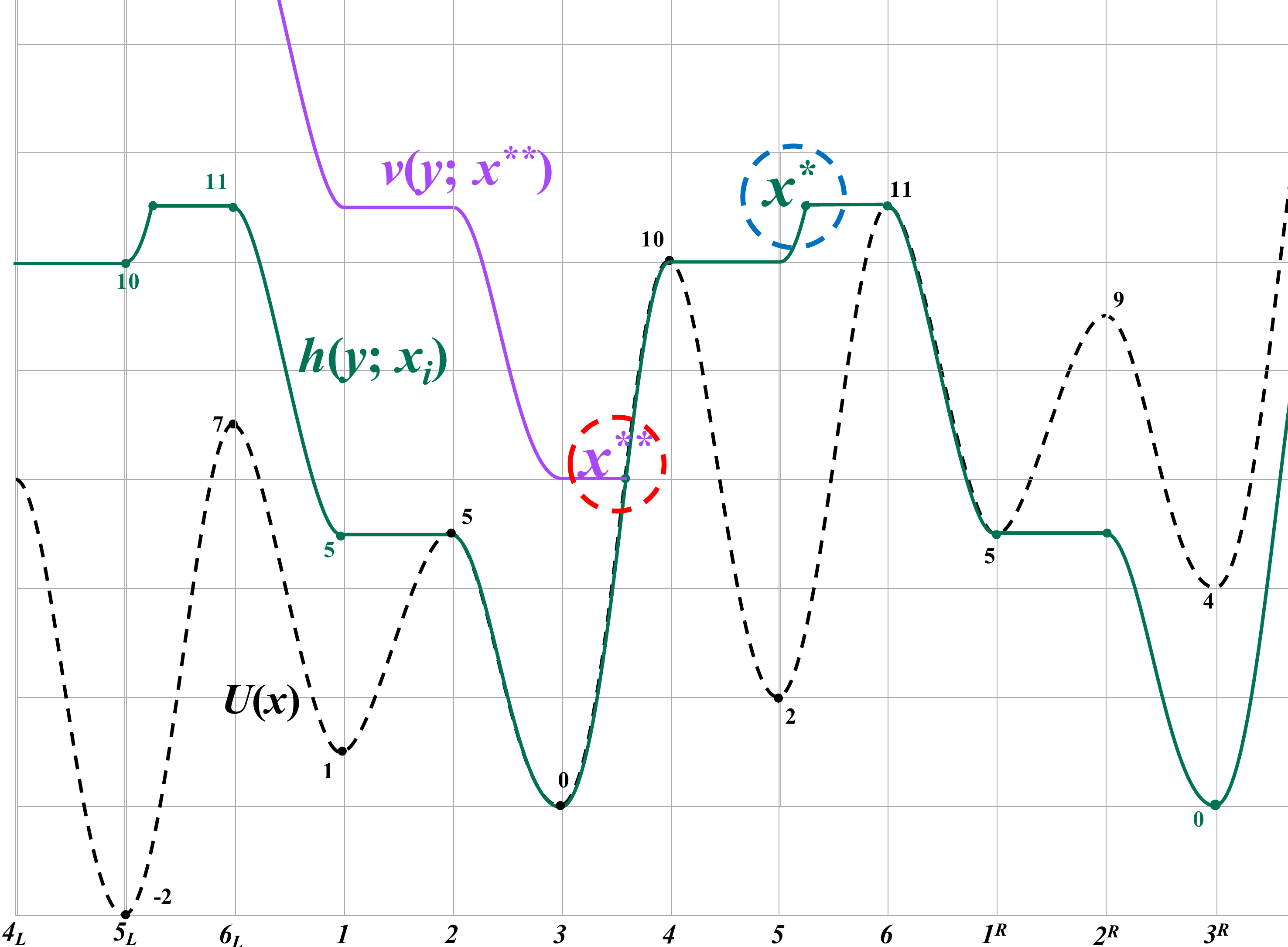}
\caption{The comparison between the   Peierls barrier $h(y;x_2)$ and a general 
Ma\~n\'e potential $v(y; x^{**}).$ The  dashed black line is the skew periodic potential $U(x)$ with three local minimums $1,0,2$ in $\bS^1=[6_L,6]$. The solid green  line is a Peierls barrier $h(y;x_2)$ starting from a local minimum $x_2=3$, which is periodic with one constant-cut from above and only one non-differential point $x^*$. The left slop is larger than the right slop at $x^*$, so $h(x^*;x_2)$ satisfies viscosity solution test.  The solid purple  line starting from $x^{**}$ is the left half of the Ma\~n\'e potential $v(y; x^{**})$, which has an additional constant-cut from below at the nondifferential point $x^{**}$. The left slop is smaller than the right slop at $x^{**}$, so $v(x^{**};x^{**})$ violates the viscosity supersolution test (i.e., violates the entropy condition).  }\label{fig_p_m}
\end{figure}
 
As a byproduct, we also characterize the shape of the Ma\~n\'e potential $v(y; x^{**})$ and explain why we do not use the Ma\~n\'e potential to construct a global energy landscape.
\begin{lem}\label{lem:mane}
Let the Ma\~n\'e potential $v(y;x^{**})$ be defined in \eqref{value} and assume $x^{**}$ is not a critical point of $U(x)$.  
\begin{enumerate}[(i)]
\item The Ma\~n\'e potential $v(y;x^{**})$ is Lipschitz continuous and periodic;
\item The starting point $x^{**}$ must be a non-differential point where either   a constant is connected to an increasing function  or a decreasing function is connected to a constant. That is to say, $v(y;x^{**})$ is a $C^1$ function  cut off  at least once  by a constant zero from below.
\item Another possible non-differential point is same as that for the Peierls barrier;
\item  $v(y; x_i)$ is the maximal Lipschitz continuous viscosity subsolution to HJE
\begin{equation}\label{HJE_v}
 H(v'(y),y) = v'( v'-U' )=0, \quad y\in \bS^1 
\end{equation}
satisfying $v(x^{**}; x^{**})=0$,
but it does not satisfy the viscosity supersolution test at $x^{**}$. In other words, $\rho(y) := v'(y;x^{**})$ is not a stationary  entropy shock at $x^{**}$ to the corresponding   Burgers transport equation
\begin{equation}\label{bb}
 \pt_t \rho + \pt_y(\rho^2) - \pt_y(U' \rho)=0.
\end{equation}
\end{enumerate}
\end{lem}
\begin{proof}
First, 
    we consider the $i$-th well of $U(x)$ containing the starting point $x^{**}$. Assume $x^{**}\in (x_i, x_{i+\frac12})$, then 
\begin{equation}
v(y;x^{**}) := \left\{ \begin{array}{cc}
U(y)-U(x_i), & y\in (x_{i-\frac12}, x_i);\\
\max\{ U(y)-U(x^{**}), 0  \}, & y\in (x_i, x_{i+\frac12}).
\end{array} \right.
\end{equation}
This means at $x^{**}$, a constant $0$ is connected to an increasing function $U(y)-U(x^{**})$. Similarly, if $x^{**}\in (x_{i-\frac12}, x_{i})$, we obtain at   $x^{**}$, a decreasing function $U(y)-U(x^{**})$ is connected to a constant. Thus proves conclusion (ii).

Second, for $y$ outside $i$-th well, the construction is the same as the Peierls barrier. Thus conclusion (i) and (iii) follow.

Third, we only need to verify the viscosity solution test   at the   non-differential point $x^{**}$. If the non-differential point is a constant $0$ connected to an increasing function $U(y)-U(x^{**})$, then $D^+ v(x^{**})=\emptyset, \, D^- v(x^{**})= \{q; 0\leq q \leq v'(x^{**}_+)=U'(x^{**})\}$. Then it's easy to verify $q( q-U'(x) )\leq 0$ satisfies the subsolution condition but does not satisfy the viscosity supersolution condition.
Again, from Step 4 in the proof of Proposition \ref{prop_h}, we know  the Ma\~n\'e potential $v(y;x^{**})$   is the maximal Lipschitz continuous viscosity subsolution to  \eqref{HJE_v}.

Last, we take $\rho(y) = v'(y;x^{**})$, and then the solution to \eqref{HJE_v} is equivalent to the stationary solution $\rho(y)$ to Burgers transport equation \eqref{bb}. The stationary shock solution $\rho(y)$ at the  jump point $x^{**}$, with the left limit $\rho_-$ and right limit $\rho_+$,  satisfies 
$$\rho_+^2- U'(x^{**}) \rho_+ = \rho_-^2 - U'(x^{**})\rho_-.$$ The entropy condition for a shock solution is that for any convex entropy function $\eta(\rho)$,
\begin{equation}\label{entropy}
\pt_t \eta(\rho) + \pt_y \bbs{\Phi(\rho) - U' \eta(\rho)} + \bbs{\eta(\rho)- \rho \eta'(\rho)} U'' \leq 0
\end{equation}
in the distribution sense. Here $\Phi'(\rho)=2 \rho \eta'(\rho)$. For scalar equations, one can just take $\eta(\rho)=\rho^2$ and thus $\Phi(\rho)=\frac{4}{3}\rho^3$.   Then the entropy condition \eqref{entropy} for the stationary shock $\rho(y)$ at $x^{**}$ becomes
\begin{equation}
\Phi(\rho_+) - U'(x^{**}) \eta(\rho_+) - \bbs{\Phi(\rho_-) - U'(x^{**}) \eta(\rho_-)} = \frac13(\rho_+ - \rho_-)^3 \leq 0,
\end{equation}
 which implies $\rho(y)$ only has jump discontinuity at $x^{**}$ and the left limit $\rho_-$ is larger than the right limit $\rho_+$. Back to $v(y;x^{**})$, the entropy condition  is violated at $x^{**}$ since $v'(x^{**}_+; x^{**})> v'(x^{**}_-; x^{**}).$ This entropy condition violation argument is equivalent to the violation of the viscosity supersolution condition for the Ma\~n\'e potential $v(y;x^{**})$.
 \end{proof}
 
 \section{Remark on the weak KAM solutions of positive type}
One can also   define a weak KAM solution of positive type, the only difference in the theory is a time direction. That is to say, the calibrated curve is defined on $[0,+\8)$ and for any $0\leq a<b $ the least action is achieved
\begin{equation}
u_+(\vg(b) )- u_+(\vg(a))= \int_a^b L(\dot{\vg}, \vg) \ud t.
\end{equation}
Moreover, the weak KAM solution of positive type $u_+$ can be equivalently characterized as a invariant solution to  the Lax-Oleinik semigroup $S^+_t$ associated with the dynamic HJE \cite{evans2008weak} 
\begin{equation}\label{HJE+}
\pt_t u - H(\pt_x u(x), x)=0, \quad x\in \bS^1, \quad u(x,0)=u_0.
\end{equation}
The viscosity solution to \eqref{HJE+} is represented  as the backward semigroup $S^+_t$, i.e., for $t\geq 0$,
\begin{equation}
(S^+_t u_0)(y):=   \sup_x\bbs{u_0(x) - \inf_{\vg; \gamma(0)=y, \, \vg(t)=x} \int_0^t L(\dot{\vg}, \vg) \ud \tau}, \quad y\in \bS^1.
\end{equation}
Then $u_+$ is the invariant solution of $S^+_y$ satisfying 
\begin{equation}
u_+(y)=(S^+_t u_+)(y)  
\end{equation}
and it is a viscosity solution to the stationary HJE \cite[Theorem 3.1]{evans2008weak}
\begin{equation}
- H(\pt_x u_+(x), x)=0, \quad x\in \bS^1.
\end{equation}

It worth noting that the weak KAM solution of positive type $u_+(x)$ is \textit{not} same as the negative ones in general. But the weak KAM solution of positive type $u_+(x)$ can be constructed via the negative type ones  with a time reversed Hamiltonian.  Precisely, define the time reversed Hamiltonian as $\hat{H}(p,x)= H(-p, x)$ and the corresponding time reversed Lagrangian is $\hat{L}(s,x)=L(-s,x)$. Then it is easy to see, the weak KAM solution of negative type, denoted as $\hat{u}_-(x)$, for the HJE $$\hat{H}(\pt_x \hat{u}_-(x), (x) )=0$$ satisfies the relation
\begin{equation}\label{neg}
-\hat{u}_-(x) = u_+(x),
\end{equation}
where $u_+(x)$ is a weak KAM solution of positive type for the HJE 
$$-{H}(\pt_x {u}_+(x), (x) )=0.$$ Apparently, at non-differential points, the viscosity solution test is different for the aboves two stationary HJEs.
For instance, in our $\bS^1$ Langevin dynamics example,  $\hat{H}(p,x)=p(p+U')$, and thus by Proposition \ref{prop_local},  $\hat{u}_-(x)$ can be expressed as the local trimming from above of $-U$, i.e.,
\begin{equation}
\hat{u}_-(x) = \min \{ -U, \text{const} \} \,\, \text{ for } x\in (x_{i-\frac12}, x_i) \text{ or } x\in (x_{i}, x_{i+\frac12}).
\end{equation}
In terms of $u_+$, this is a local trimming from below of $U$, i.e.,
\begin{equation}
u_+(x) = -\hat{u}_-(x) = -\min \{ -U, \text{const} \} = \max \{U, - \text{const}\}.
\end{equation}
However, when the potential $U$ is periodic, i.e., the Langevin dynamics is reversible, no cut-off from above/below is performed, and thus  the positive type weak KAM solution $u_+(x)$ given by \eqref{neg} is same as the negative type $u_-(x)$ constructed via the large deviation principle(see Corollary \ref{cor_kam}). Indeed, $\hat{H}(p,x)=p(p+U')$ and $\hat{u}_-(x) = -U(x)$ is a weak KAM solution of the negative type associated with $\hat{H}(\pt_x \hat{u}_-(x), (x) )=0$, which is actually solved in the classical sense. Thus $u_+(x) = U(x)=u_-(x).$  This argument is no longer true for the irreversible process, i.e., $U(x)$ is not periodic.

 \section{Freidlin-Wentzell's variational formula}\label{app2}
 In this section, we give a coarse grained Markov chain interpretation for Freidlin-Wentzell's variational formula \eqref{FW_W}.

 To study the multi-well exit problem, the essential idea follows Kolmogorov's construction of Markov chain induced by the continuous process $X_t$ in \eqref{LV}. Denote the collection of all the local minimums as $\Gamma:=\{x_i,i=1,\cdots, k\}$. Denote the stopping time $\tau_i:=\inf \{t> \tau_{i-1}; X_t\in \Gamma\backslash \tilde{X}_{\tau_{i-1}}\}$ and   $\tilde{X}_t := X_{\tau_{i-1}}\in \Gamma$ for $t\in [\tau_{i-1}, \tau_i)$ is defined by the sequence of $\tau_i, \, i=0,1,\cdots$. This  is the induced continuous time Markov chain on $\Gamma$. The transition probability for $\tilde{X}$ can be approximated by the large deviation principle for exit problems 
 \begin{equation}
 P_{i, i+1} := \bP\{\tilde{X}=i+1| \tilde{X} = i\} \approx c e^{-\frac{U_{i+\frac12}-U_i}{\eps}}. 
 \end{equation}
 Similarly, define $P_{i, i-1} \propto e^{-\frac{U_{i-\frac12}-U_i}{\eps}}$. This defines an approximated $Q$-process with transition probability matrix $(P_{ij})$. Then the invariant distribution $\nu_i^\eps$, $i=1,\cdots, k$ satisfies
 \begin{equation}
 \nu_i^\eps P_{i,i-1} + \nu_i ^\eps P_{i, i+1} = \nu_{i-1}^\eps P_{i-1,i} + \nu_{i+1}^\eps P_{i+1, i}.
 \end{equation}
One can directly verify the closed formula for $\nu^\eps_i$ is given by
\begin{equation}\label{eq_a3}
\nu^\eps_i = \sum_{j=1}^k e^{-\frac{ \tilde{h}_R(x_i; x_{j+1}) \, + \,\tilde{h}_L(x_i; x_{j}) }{\eps}}.
\end{equation} 
Indeed, this formula is  the principal left eigenvector $\nu^T Q=0$  of a cyclic stochastic matrix  
$$Q=\left( \begin{array}{ccccc}
-\!a_1\!\!-\!b_1 & a_1 &  &   & b_1 \\ 
b_2 & -\!a_2\!\!-\!b_2 & a_2 &   &  \\ 
 & \ddots & \ddots & \ddots  &   \\
 &  & \ddots  &  \ddots  & a_{k\!-1} \\ 
a_k &  &   & b_k & \,-\!a_k\!\!-\!b_k
\end{array}  \right), \qquad \nu_i = \sum_{j=1}^k \,\, \prod_{\ell = i-1}^{i+j-1} b_\ell \prod_{m=i+1}^{i+j-1} a_m,$$
where $a_i = P_{i,i+1}$ and $b_i = P_{i,i-1}$ with $k$-periodic index.

 Then as $\eps \to 0$, by the Laplace principle for \eqref{eq_a3}, 
we have
\begin{equation}\label{wi_app}
-\eps \log \nu_i^\eps \to W_i =  \min_{j=1,\cdots,k} (\tilde{h}_R(x_i; x_{j+1}) \, + \,\tilde{h}_L(x_i; x_{j})), \quad i=1, \cdots, k.
\end{equation}
This is the variational formula for boundary data $W_i$ in \eqref{wi}.

Based on the invariant measure $\nu^\eps_i$, $i=1,\cdots,k$ for the induced Markov chain, one can recover the original invariant measure $\pi_\eps$ by the celebrated ergodic result by \textsc{Khasminskii} \cite{khas1960ergodic}. This, together with boundary data $W_i$, $i=1, \cdots, k$, can recovers Freidlin-Wentzell's variational formula \eqref{FW_W}.

\bibliographystyle{alpha}
\bibliography{KAM_bib}

\end{document}